\documentclass[12pt,a4paper]{amsart}
\usepackage{mathtools}
\usepackage{xcolor}
\usepackage{amsmath,amstext,amssymb,amsthm,amsfonts}
\usepackage{graphicx}
\usepackage[all]{xy} \xyoption{poly}
\usepackage{ mathrsfs }
\newcommand{\nc}{\newcommand}

\nc{\ad}{\operatorname{ad}}
\nc{\ev}{\operatorname{ev}}
\nc{\tr}{\operatorname{tr}}

\nc{\Gr}{\mathscr{K}}
\nc{\rGr}{\operatorname{rGr}}
\nc{\atyp}{\operatorname{atyp}}
\nc{\tp}{\operatorname{top}}
\nc{\rank}{\operatorname{rank}}
\nc{\corank}{\operatorname{corank}}
\nc{\codim}{\operatorname{codim}}
\nc{\sdim}{\operatorname{sdim}}
\nc{\mult}{\operatorname{mult}}
\nc{\ds}{\operatorname{ds}}
\nc{\tail}{\operatorname{tail}}
\nc{\howl}{\operatorname{howl}}
\nc{\triv}{\operatorname{triv}}
\nc{\spn}{\operatorname{span}}
\nc{\Sym}{\operatorname{Sym}}
\nc{\Core}{\operatorname{Core}}
\nc{\id}{\operatorname{id}}
\nc{\Id}{\operatorname{Id}}
\nc{\Ree}{\operatorname{Re}}
\nc{\hi}{\operatorname{hi}}
\nc{\htt}{\operatorname{ht}}
\nc{\at}{\operatorname{at}}
\nc{\str}{\operatorname{str}}
\nc{\Iso}{\operatorname{Iso}}
\nc{\Ker}{\operatorname{Ker}}
\nc{\rker}{\operatorname{rKer}}
\nc{\im}{\operatorname{Im}}
\nc{\osp}{\mathfrak{osp}}
\nc{\sgn}{\operatorname{sgn}}
\nc{\F}{\operatorname{F}}
\nc{\Mod}{\operatorname{Mod}}
\nc{\DS}{\operatorname{DS}}
\nc{\Soc}{\operatorname{Soc}}
\nc{\Inj}{\operatorname{Inj}}
\nc{\Hom}{\operatorname{Hom}}
\nc{\End}{\operatorname{End}}
\nc{\supp}{\operatorname{supp}}
\nc{\smult}{\operatorname{smult}}
\nc{\Dyn}{\operatorname{Dyn}}
\nc{\Card}{\operatorname{Card}}
\nc{\Ann}{\operatorname{Ann}}
\nc{\Arc}{\operatorname{Arc}}
\nc{\arc}{\operatorname{arc}}
\nc{\Ind}{\operatorname{Ind}}
\nc{\Coind}{\operatorname{Coind}}
\nc{\wt}{\operatorname{hwt}}
\nc{\hwt}{\operatorname{wt}}
\nc{\rk}{\operatorname{rank}}
\nc{\ch}{\operatorname{ch}}
\nc{\sch}{\operatorname{sch}}
\nc{\mdim}{\operatorname{mdim}}
\nc{\Stab}{\operatorname{Stab}}
\nc{\hull}{\operatorname{hull}}

\nc{\Irr}{\operatorname{Irr}}
\nc{\Spec}{\operatorname{Spec}}
\nc{\Res}{\operatorname{Res}}
\nc{\res}{\operatorname{res}}
\nc{\Aut}{\operatorname{Aut}}
\nc{\Ext}{\operatorname{Ext}}
\nc{\Prec}{\operatorname{Prec}}
\nc{\Fract}{\operatorname{Fract}}
\nc{\gr}{\operatorname{gr}}
\nc{\vol}{\operatorname{vol}}
\nc{\diag}{\operatorname{diag}}
\nc{\deff}{\operatorname{def}}
\nc{\core}{\operatorname{core}}
\nc{\HC}{\operatorname{HC}}
\nc{\Ch}{\operatorname{Ch}}
\nc{\Sch}{\operatorname{Sch}}
\nc{\dpth}{\operatorname{dpth}}
\nc{\Pol}{\operatorname{Pol}}

\nc{\pari}{\operatorname{par}}
\nc{\pos}{\operatorname{pos}}

\nc{\Cl}{\mathcal{C}\ell}

\nc{\wdchi}{\widetilde{\chi}}
\nc{\wdH}{\widetilde{H}}
\nc{\wdN}{\widetilde{N}}
\nc{\wdM}{\widetilde{M}}
\nc{\wdO}{\widetilde{O}}
\nc{\wdR}{\widetilde{R}}

\nc{\wdV}{\widetilde{V}}

\nc{\wdC}{\widetilde{C}}

\nc{\zero}{\operatorname{zero}}
\nc{\nonzero}{\operatorname{nonzero}}

\nc{\Nonzero}{\operatorname{Nonzero}}

\nc{\diam}{\operatorname{diam}}

\nc{\Obj}{\operatorname{Obj}}
\nc{\Dglie}{\operatorname{{\mathcal D}glie}}
\nc{\Fin}{\operatorname{{\mathcal F}in}}
\nc{\pr}{\operatorname{pr}}
\nc{\Adm}{\operatorname{\mathcal{A}dm}}

\nc{\fg}{\mathfrak g}

\nc{\Sg}{{\cS(\fg)}}
\nc{\Shg}{{\cS(\fhg)}}
\nc{\Ug}{{\cU(\fg)}}
\nc{\Uhg}{{\cU(\fhg)}}
\nc{\Sh}{{\cS(\fh)}}
\nc{\Uh}{{\cU(\fh)}}
\nc{\Uhh}{{\cU(\fhh)}}
\nc{\Zg}{{{\mathcal{Z}}(\fg)}}

\nc{\Vir}{{\mathcal{V}ir}}
\nc{\NS}{{\mathcal{N}S}}

\nc{\tZg}{{\widetilde{\mathcal Z}({\mathfrak g})}}
\nc{\Zk}{{\mathcal Z}({\mathfrak k})}

\nc{\Up}{{\mathcal U}({\mathfrak p})}
\nc{\Ah}{{\mathcal A}({\mathfrak h})}
\nc{\Ag}{{\mathcal A}({\mathfrak g})}
\nc{\Ap}{{\mathcal A}({\mathfrak p})}
\nc{\Zp}{{\mathcal Z}({\mathfrak p})}
\nc{\cR}{\mathcal R}
\nc{\cS}{\mathcal S}
\nc{\cP}{\mathcal P}
\nc{\cT}{\mathcal{T}}
\nc{\CC}{\mathcal C}
\nc{\cA}{\mathcal A}
\nc{\cV}{\mathcal V}
\nc{\cU}{\mathcal U}
\nc{\cZ}{\mathcal Z}
\nc{\cM}{\mathcal M}
\nc{\cL}{\mathcal L}
\nc{\cF}{\mathcal F}

\nc{\cB}{\mathcal{B}}

\nc{\fo}{\mathfrak o}

\nc{\fa}{\mathfrak a}

\nc{\CO}{\mathcal O}
\nc{\CR}{\mathcal R}

\nc{\cK}{\mathcal{K}}
\nc{\cW}{\mathcal{W}}
\nc{\bM}{\mathbf{M}}
\nc{\bL}{\mathbf{L}}
\nc{\bN}{\mathbf{N}}

\nc{\zq}{\mathpzc q}

\nc{\fl}{\mathfrak l}
\nc{\fn}{\mathfrak n}
\nc{\fm}{\mathfrak m}
\nc{\fp}{\mathfrak p}
\nc{\fh}{\mathfrak h}
\nc{\ft}{\mathfrak t}
\nc{\fk}{\mathfrak k}
\nc{\fb}{\mathfrak b}
\nc{\fs}{\mathfrak s}
\nc{\fB}{\mathfrak B}

\nc{\vareps}{\varepsilon}
\nc{\varesp}{\varepsilon}
\nc{\veps}{\varepsilon}

\nc{\fsl}{\mathfrak{sl}}
\nc{\fgl}{\mathfrak{gl}}
\nc{\fso}{\mathfrak{so}}
\nc{\fosp}{\mathfrak{osp}}
\nc{\fsp}{\mathfrak{sp}}
\nc{\fq}{\mathfrak q}
\nc{\fsq}{\mathfrak{sq}}
\nc{\fpsq}{\mathfrak{psq}}
\nc{\fpq}{\mathfrak{pq}}

\nc{\fhg}{\hat{\fg}}
\nc{\fhn}{\hat{\fn}}
\nc{\fhh}{\hat{\fh}}
\nc{\fhb}{\hat{\fb}}
\nc{\hrho}{\hat{\rho}}

\nc{\hsl}{\hat{\fsl}}
\nc{\fpo}{\mathfrak{po}}
\nc{\dirlim}{\underset{\rightarrow}{\lim}\,}
\nc{\nen}{\newenvironment}
\nc{\ol}{\overline}
\nc{\ul}{\underline}
\nc{\ra}{\rightarrow}
\nc{\lra}{\longrightarrow}
\nc{\Lra}{\Longrightarrow}
\nc{\bo}{\bar{1}}
\nc{\Lla}{\Longleftarrow}

\nc{\Llra}{\Longleftrightarrow}

\nc{\thla}{\twoheadleftarrow}

\nc{\lang}{(}
\nc{\rang}{)}

\nc{\hra}{\hookrightarrow}

\nc{\iso}{\overset{\sim}{\lra}}

\nc{\ssubset}{\underset{\not=}{\subset}}

\nc{\vac}{|0\rangle}

\nc{\simka}{{\ \scriptscriptstyle _{{\sim}}^\text{\tiny{k}}\ }}

\nc{\Thm}[1]{Theorem~\ref{#1}}
\nc{\Prop}[1]{Proposition~\ref{#1}}
\nc{\Lem}[1]{Lemma~\ref{#1}}
\nc{\Cor}[1]{Corollary~\ref{#1}}
\nc{\Conj}[1]{Conjecture~\ref{#1}}
\nc{\Claim}[1]{Claim~\ref{#1}}
\nc{\Defn}[1]{Definition~\ref{#1}}
\nc{\Exa}[1]{Example~\ref{#1}}
\nc{\Rem}[1]{Remark~\ref{#1}}
\nc{\Note}[1]{Note~\ref{#1}}
\nc{\Quest}[1]{Question~\ref{#1}}
\nc{\Hyp}[1]{Hypoth\`ese~\ref{#1}}
\nen{thm}[1]{\label{#1}{\bf Theorem.\ } \em}{}
\nen{prop}[1]{\label{#1}{\bf Proposition.\ } \em}{}
\nen{lem}[1]{\label{#1}{\bf Lemma.\ } \em}{}
\nen{cor}[1]{\label{#1}{\bf Corollary.\ } \em}{}
\nen{conj}[1]{\label{#1}{\bf Conjecture.\ } \em}{}

\nen{claim}[1]{\label{#1}{\bf Claim.\ } \em}{}

\nen{defn}[1]{\label{#1}{\bf Definition.\ } }{}
\nen{exa}[1]{\label{#1}{\bf Example.\ } }{}
\nen{rem}[1]{\label{#1}{\em Remark.\ } }{}
\nen{note}[1]{\label{#1}{\em Note.\ } }{}
\nen{exer}[1]{\label{#1}{\em Exercise.\ } }{}
\nen{sket}[1]{\label{#1}{\em Sketch of proof.\ } }{}
\nen{quest}[1]{\label{#1}{\bf Question.\ } \em}{}

\nen{hyp}[1]{\label{#1}{\bf Hypoth\`ese.\ } \em}{}
\begin{document}
\setcounter{section}{-1}
\setcounter{tocdepth}{1}

\title[]{On the character ring of  a quasireductive Lie superalgebra}
\author{ M. Gorelik}

\address{Weizmann Institute of Science, Rehovot, Israel}
\email{maria.gorelik@weizmann.ac.il}

%
%

\date{}

\begin{abstract} 
We study character rings of quasireductive Lie superalgebras
and give a new proof of the Sergeev-Veselov  theorem
describing the character rings of finite-dimensional Kac-Moody superalgebras. 
\end{abstract}


\medskip


\maketitle
\rightline{{\em
Order is the key to all problems.}}
\rightline{
Alexandre Dumas,}
\rightline{The Count of Monte Cristo}

\section{Introduction}
One of the classical results of the representation theory of  complex semisimple Lie algebras can be formulated as follows (see e.g.~\cite{SV}):

{\em The character map induces an isomorphism between
the representation ring of finite-dimensional
modules of a complex semisimple Lie algebra $\fg$ and $W$-invariants in
the group ring $\mathbb{Z}[P_0]$ where $W$ is the Weyl group and $P_0$
is the weight lattice of $\fg$. }

The ``representation ring''  $\cK(\fg)$ of a Lie algebra $\fg$ is, by definition, the Grothendieck group
of the category of finite-dimensional modules, that is the quotient
of a free abelian group with generators given by all isomorphism classes
of finite-dimensional $\fg$-modules by the subgroup generated by
$[N]-[M]-[N/M]$ for each module $N$ and its submodule $M$.
We will use the same definition for Lie superalgebras and  denote by $\cK_{\pm}(\fg)$ the quotient of $\cK(\fg)$ by the relations
$[N]=\pm [\Pi N]$, where $\Pi$ stands for the parity change functor.
 In~\cite{SV} Sergeev and Veselov described the ring $\cK_-(\fg)$ for the 
finite-dimensional Kac-Moody superalgebras and some related algebras.
For $Q$-type superalgebras
Reif described the subring corresponding to integral and half-integral weights  
 of $\cK_+(\fg)$ in~\cite{Reif}.

%

Let $\fg$ be a quasireductive Lie superalgebra (see~\cite{Sqred},\cite{Maz}): this is a finite-dimensional
Lie superalgebra with a reductive even part $\fg_{\ol{0}}$
which acts semisimply on the odd part $\fg_{\ol{1}}$. Let $\ft$
be a Cartan subalgebra of $\fg_{\ol{0}}$ and $\fh:=\fg^{\ft}$ be a Cartan subalgebra of $\fg$. If $\fg'$ is a subalgebra of $\fg$, the restriction functor 
$\Res^{\fg}_{\fg'}$ induces the ring  homomorphisms $\res^{\fg}_{\fg'}$
from $\cK(\fg)$  to $\cK(\fg')$
and from $\cK_{\pm}(\fg)$ to $\cK_{\pm}(\fg')$. 
All three homomorphisms $\res^{\fg}_{\fh}$
are injective as well as the homomorphism
$\res^{\fg}_{\ft}:\cK_{+}(\fg)\to\cK_{+}(\ft)$.
The ring $\cK(\ft)$ can be naturally identified with the ring
$\mathbb{Z}[\ft^*]\otimes_{\mathbb{Z}} \mathbb{Z}[\xi]$
where $\mathbb{Z}[\ft^*]$ is the group ring of $\ft^*$
and  $\xi$ a formal variable satisfying $\xi^2=1$.
In this paper we study
 the image of $\res^{\fg}_{\ft}:\cK(\fg)\to\cK(\ft)$, which we denote
by $\Ch(\fg)$ and  the images of $\res^{\fg}_{\ft}:\cK_{\pm}(\fg)\to\cK_{\pm}(\ft)$, which we denote
by $\Ch_{\pm}(\fg)$. The image of $\cK(\fg)$ in $\cK(\fh)$ is studied in~\cite{GSS}.
By above, $\cK_+(\fg)\cong \Ch_+(\fg)$; if $\fh=\ft$, then
  $\cK_-(\fg)\cong \Ch_-(\fg)$  (this holds in
the cases studied by Sergeev and Veselov in~\cite{SV}).

We fix a triangular decomposition of
$\fg_{\ol{0}}$ and denote by $\pi$ the
set of simple roots, by $P^+(\pi)$
 the set of $\fg_{\ol{0}}$-dominant weights
and  by $P_0$ the $\mathbb{Z}$-span of $P^+(\pi)$.
Let $W\subset GL(\ft^*)$ be the Weyl group. We set
$\mathbb{Z}[P_0;\xi]:=\mathbb{Z}[P_0]\otimes_{\mathbb{Z}} \mathbb{Z}[\xi]$.
Since all weights of finite-dimensional $\fg$-module lie in $P_0$,
$\Ch(\fg)$ is a subring in the ring $R(P_0):=\Ch(\fh)\cap \mathbb{Z}[P_0;\xi]$
(one has $R(P_0)=\mathbb{Z}[P_0;\xi]$ if  $\fh=\ft$).

The $W$-action $we^{\lambda}:=e^{w\lambda}$ naturally extends to
the $W$-action on $R(P_0)$.

We denote by $\Delta$  the root system of $\fg$ and by $\Delta_{iso}$
the set of roots $\beta$ such that for some $e_{\pm}\in \fg_{\pm\beta}$
the subalgebra spanned by $e_{\pm}, h_{\beta}:=[e_+,e_-]$ is isomorphic to
 $\fsl(1|1)$.  Recall
that any element in $R(P_0)\subset \mathbb{Z}[P_0;\xi]$ takes the form
$\sum_{\nu} m_{\nu} e^{\nu}$ with $m_{\nu}\in \mathbb{Z}[\xi]$.
We set
\begin{equation}\label{AP}\begin{array}{l}
\cA(\fg):=\bigcap_{\beta\in\Delta_{iso}}\{\displaystyle\sum_{\nu} m_{\nu} e^{\nu}\in R(P_0)^W\ |\   \langle\nu,h_{\beta}\rangle\not=0\ \Longrightarrow\ 
 \displaystyle\sum_{i\in\mathbb{Z}} 
(-\xi)^i  m_{\nu+i\beta}=0\}.
\end{array}
 \end{equation}

Considering the homomorphisms $\res^{\fg}_{\fgl(1|1)}:\cK(\fg)\to\cK(\fgl(1|1))$ we obtain
$$\Ch(\fg)\subset \cA(\fg).$$
In this paper we show that $\Ch(\fg)= \cA(\fg)$ 
if $\fg$ is a finite-dimensional Kac-Moody superalgebra
or a $Q$-type superalgebra. From~\cite{SV} it follows that $\Ch(\fg)\not=\cA(\fg)$
for $\fg=\mathfrak{psl}(2|2)$. In addition, we show that
if $\fg$ is a $Q$-type superalgebra or a finite-dimensional Kac-Moody superalgebra
$\fg\not=\fgl(1|1)$, then the ring $\Ch(\fg)$ admits a ``short''
$\mathbb{Z}$-basis, see~\ref{short} for definition.
If $\fg$ is semisimple, this short basis is
$\{b_{\lambda},\xi b_{\lambda}\}_{\lambda\in P^+(\pi)}\ $ where
$b_{\lambda}=\sum_{\nu\in W\lambda} e^{\nu}$.

The formula $\Ch(\fg)=\cA(\fg)$ implies that
$\Ch_{\pm}(\fg)=\cA_{\pm}(\fg)$, where $\cA_{\pm}(\fg)$ are the images
of $\cA(\fg)$ under the evaluations $\xi\mapsto\mp 1$.
The formula $\Ch_-(\fg)=\cA_-(\fg)$ is equivalent to the Sergeev-Veselov formula.


The formula $\Ch(\fg)=\cA(\fg)$ for the
finite-dimensional Kac-Moody superalgebras (resp., $Q$-type) 
can be deduced from the results of~\cite{SV} (resp.,~\cite{Reif}), but
we  give an alternative proof of these results. By contrast to the proof in~\cite{SV},
our proof does not require 
knowlegde of any characters, but is based on the existence of ``short bases''.
Our main tool is the standard partial order  on $\ft^*$.
 Instead of the distinguished base of simple roots 
used in~\cite{SV} we use so-called ``mixed bases''. 
The distinguished bases contain
at most one odd root whereas
 the mixed bases contain the maximal possible 
number of odd roots. The mixed bases are useful for Kac-Wakimoto character 
formulae. For the $\osp$-case the mixed bases were used
 for the description of characters in~\cite{GS}. 
For $\fg\not=\fgl(n|n)$ our proof is based on the following simple lemma.

\begin{lem}{propmainstep0}
Let $\fg$ be a quasireductive Lie superalgebra with a base $\Sigma$ satisfying  the following properties:
\begin{itemize}
\item[(Pr1)]  $-\mathbb{R}_{\geq 0}\Delta^+\cap \mathbb{R}_{\geq 0}P^+(\pi)=\{0\}$.
\item[(Pr2)] {\em For each $\eta\in P^+(\pi)\setminus P^+(\Delta^+)$
there exists  $\beta\in\Sigma\cap\Delta_{iso}$ 
such that }
$$\langle \eta,h_{\beta}\rangle\not=0\ \ \text{ and }\ \ \eta-\beta\not\in P^+(\pi).$$
\end{itemize}
Then $\Ch(\fg)$ admits a ``short basis'' and 
$\Ch(\fg)= \cA(\fg)$. 
\end{lem}

It is easy to see that (Pr1) holds if $\Delta^+$ and $-\Delta^+$
are $W$-conjugated. The property (Pr2) looks rather technical, but, in fact, admits a nice interpretation in terms of odd reflections;
 each finite-dimensional Kac-Moody superalgebra admits 
bases satisfying 
\begin{equation}\label{P+Delta}
P^+(\Delta^+)=\{\lambda\in P^+(\pi)|\ \forall \beta\in\Sigma\cap\Delta_{iso}\ \
\langle\lambda,h_{\beta}\rangle\not=0\ \Longrightarrow \lambda-\beta\in P^+(\pi)\}
\end{equation}
which obviously implies (Pr2).
The mixed triangular decompositions satisfies both  (Pr1) and (Pr2) for $\fg\not=\fgl(n|n)$. For  
the $Q$-type superalgebras all triangular decompositions satisfy (Pr1), (Pr2). This establishes $\Ch(\fg)= \cA(\fg)$
for the $Q$-type superalgebras and the finite-dimensional Kac-Moody superalgebras
  $\fg\not=\fgl(n|n)$. 
 For $\fgl(n|n)$ we take a triangular decompostion satisfying (Pr1) and
 a weaker version of (Pr2).

{\bf Contents of the paper.}
In Section~\ref{ringK} we collect some general results about
the Grothendieck  rings $\cK(\fg)$ and the rings $\Ch(\fg)$ in the case when $\fg$ is quasireductive.
In particular, we show that $\Ch(\fg)$ is a subring in $\cA(\fg)$.

In Section~\ref{sect1} we define ``short bases'', 
prove the above lemma and deduce the formula $\Ch(\fg)= \cA(\fg)$
for the $Q$-type superalgebras and the finite-dimensional Kac-Moody superalgebras.

In Section~\ref{Reform} 
we describe the rings $\cA_{\pm}(\fg)$ in the fashions of~\cite{SV}
and of~\cite{Reif}. We deduce
the formula $\Ch(\fp_n)=\cA(\fp_n)$ from the results of~\cite{IRS}.

In Appendix we recall the criteria of dominance for  the finite-dimensional Kac-Moody superalgebras and obtain
the formula~(\ref{P+Delta}) in~\Cor{coro}.

{\bf Acknowledgments.}
The author was supported by  ISF Grant 1957/21. 
The author is grateful to  Vladimir Hinich for numerous helpful suggestions and
to Inna Entova-Aizenbud,  Shifra Reif, Alex Sherman and  Vera Serganova for 
stimulating discussions. The author   is grateful to Shay Kinamon Kerbis and Lior-David Silberberg for their corrections and comments on the earlier
version of the manuscript.

{\bf Index of definitions and notation}
Throughout the paper the ground field is $\mathbb{C}$; 
$\mathbb{N}$ stands 
 for the set of non-negative integers. We denote by $\Pi$ the parity change functor. 

\begin{center}
\begin{tabular}{lcl}
$\ft$, $\pi$, $P^+(\pi)$, $P_0$,  $\fh$, $\Delta$, $W$, $s_{\alpha}$, partial order $>$, a base & & \ref{tri} \\

$\Delta_{iso}$, $h_{\beta}$ & &  \ref{setio}\\

$C_{\lambda}$, $L(\lambda)$, $P^+(\Delta^+)$ & & \ref{Clambda}\\

$\cK(\CC)$, $\cK_{\pm}(\CC)$, $\cK(\fg)$, $\cK_{\pm}(\fg)$, $\psi_{\pm}$, $\res^{\fg}_{\fg'}$ & & \ref{notat2} \\

$I_1$, $\supp$ & & \ref{basisK} \\

$\mathbb{Z}[P_0;\xi]$,  $\psi_{\pm}$, $\ch_{\xi}$, $\ch$, $\sch$, $\Ch(\fg)$, $\Ch_{\pm}(\fg)$ & & \ref{RingCh}\\ 

$R(P_0)$, $R_{\pm}(P_0)$, $\underset{R(P_0)}{\times}$  & & \ref{RP0}\\

$\cA(\fg)$, $\cA_{\pm}(\fg)$ & & \ref{RingsACh}\\
short basis, $b_{\lambda}$ & & \ref{short}\\

\end{tabular}
\end{center}

\section{Preliminary}
Throughout the paper $\fg$ is a quasireductive Lie superalgebra. 

\subsection{Triangular decompositions}\label{tri}
 Recall that $\fg_{\ol{0}}$ is a reductive Lie algebra.
We fix a Cartan subalegbra $\ft$  in $\fg_{\ol{0}}$. We  denote  by 
$\Delta_{\ol{0}}$ the set of roots of $\fg_{\ol{0}}$ and by
$\Delta$
the set of roots of $\fg$. We set $\fh:=\fg^{\ft}$.
Let $W\subset GL(\ft^*)$ be the Weyl group.

Given $h\in\ft$ with the property
$\Ree\langle\alpha,h\rangle\not=0$ for each $\alpha\in\Delta$
we have the corresponding subsets of positive roots $\Delta^+$
and $\Delta_{\ol{0}}^+$ (on which $\Ree\langle h,\alpha\rangle>0$).
For different choices of $h$ the subsets $\Delta_{\ol{0}}^+$
 may be different,
but they can be transformed to each other by the Weyl group.
We will fix one of them, $\Delta_{\ol{0}}^+$,
and consider only the subsets of positive roots $\Delta^+$ in $\Delta$,
which contain $\Delta_{\ol{0}}^+$. This choice fixes a triangular decomposition
of $\fg$, compatible with the triangular decomposition of $\fg_{\ol{0}}$,
corresponding to $\Delta_{\ol{0}}^+$. 

We  denote by $\pi$ the
set of simple roots for $\Delta_{\ol{0}}^+$.
For each $\alpha\in\pi$ we denote by $s_{\alpha}$ the corresponding reflection in $W$. Recall that $P^+(\pi)$ stands for
 the set of the highest weights of the finite-dimensional 
  simple $\fg_{\ol{0}}$-modules 
and  $P_0$ is the $\mathbb{Z}$-span of $P^+(\pi)$ ($P_0$
is a lattice if $\fg_{\ol{0}}$ is semisimple).

Our main tool is the standard partial order $>$ on
$\ft^*$ given by  $\lambda\geq \nu$ if $\lambda-\nu\in\mathbb{N}\Delta^+$.
For each $U\subset \ft^*$ we denote by $\max U$  the set of 
maximal elements in $U$.

A set $\Sigma\subset\Delta^+$ is called a {\em base} if the elements
of  $\Sigma$ are linearly independent and $\Delta^{\pm}$ lies
in $\pm\mathbb{N}\Sigma$ (where $\mathbb{N}\Sigma$
is the set of non-negative integral linear combinations of $\Sigma$);
in this case we write $\Delta^+=\Delta^+(\Sigma)$.
The root systems of Kac-Moody superalgebras, $\fq_n$, $\fp_n$ and $\fgl(m|n)$ admit a base. The compatibility condition ($\fn^+_{\ol{0}}\subset \fn^+$)
means that $\pi\subset \mathbb{N}\Sigma$.

\subsection{Set $\Delta_{iso}$}\label{setio}
We denote by $\Delta_{iso}$
the set of roots $\beta$ such that for some $e_{\pm}\in \fg_{\pm\beta}$
the subalgebra spanned by $e_{\pm}, h_{\beta}:=[e_+,e_-]$ is isomorphic to
 $\fsl(1|1)$. 
Clearly, $\Delta_{iso}=-\Delta_{iso}$.

Take  $\beta\in\Delta_{iso}$. Since $\beta\not=0$, there exists $h'\in\ft$
satisfying $\langle \beta,h'\rangle\not=0$. Taking $\fsl(1|1)$ as above we obtain
$\fsl(1|1)+\mathbb{C}h'\cong \fgl(1|1)$.

Observe that for $\fsl(1|1)$ we have $\fg_{\ol{0}}=\ft$
and $\fh=\fg$, so $\Delta=\Delta_{iso}=\emptyset$.

\subsection{}
\begin{lem}{lem-1}
For all $\lambda\in\mathbb{R}_{\geq 0}P^+(\pi)$, $w\in W$ we have $\lambda-w\lambda\in\mathbb{R}_{\geq 0}\pi$. Moreover, 
$$
\{\lambda\in P_0|\ \forall \alpha\in\pi\ s_{\alpha}\lambda\leq \lambda\}=
\{\lambda\in P_0|\ \forall w\in W\ w\lambda\leq \lambda\}=P^+(\pi).
$$

\end{lem}
\begin{proof}
Since  $\fg_{\ol{0}}$ is reductive we have
$$\begin{array}{ll}
P_0=\{\mu\in \ft^*|\ \forall \alpha\in \pi\ \ \ \mu-s_{\alpha}\mu\in\mathbb{Z}\alpha\},\\ 
P^+(\pi)=\{\mu\in \ft^*|\ \forall \alpha\in \pi\ \ \  \mu-s_{\alpha}\mu\in\mathbb{N}\alpha\}
=\{\mu\in\ft^*|\ \ \forall w\in W\ \ \lambda-w\lambda\in\mathbb{N}\pi\}.
\end{array}$$
This implies  $\lambda-w\lambda\in\mathbb{R}_{\geq 0}\pi$
for $\lambda\in\mathbb{R}_{\geq 0}P^+(\pi)$, $w\in W$ and gives
$$
P^+(\pi)\subset \{\lambda\in P_0|\ \forall \alpha\in\pi\ \ s_{\alpha}\lambda\leq \lambda\}\ \subseteq
\{\lambda\in P_0|\ \forall w\in W\ w\lambda\leq \lambda\}.
$$
Take $\lambda\in P_0$ such that $w\lambda\leq \lambda$ for all $w\in W$.
Then $s_{\alpha}\lambda\leq \lambda$ for each $\alpha\in\pi$. Since
$\lambda\in P_0$ this gives $\lambda-s_{\alpha}\lambda\in\mathbb{N}\pi$.
Thus $\lambda\in P^+(\pi)$. Hence
$$ \{\lambda\in P_0|\ \forall w\in W\ w\lambda\leq \lambda\}\subset P^+(\pi).$$
This completes the proof.
\end{proof}

\subsection{Modules $L(\lambda)$}\label{Clambda}
Let $C_{\lambda}$ be a simple
$\fh$-module where $\ft$ acts by $\lambda$ and $\sdim C_{\lambda}\geq 0$
(such module is  unique up to
a parity change $\Pi$). 
We view $C_{\lambda}$ as a $\fb$-module with the zero action of $\fn$
and denote by $L(\lambda)$  a   simple quotient of
 $M(\lambda):=\Ind^{\fg}_{\fb} C_{\lambda}$ (such simple quotient is unique).
Note that  $\dim C_{\lambda}=1$ if
$\fh=\ft$. We set
$$P^+(\Delta^+):=\{\lambda\,|\ \dim L(\lambda)<\infty\}.$$
Clearly, $P^+(\Delta^+)\subset P^+(\pi)$.

\subsection{$Q$-type superalgebras}
By $Q$-type superalgebras we mean one of the Lie superalgebra $\fq_n$
and their subquoteints $\fsq_n$,
$\mathfrak{psq}_n$, $\mathfrak{pq}_n$. These are quasireductive Lie  superalgebras: $\fg_{\ol{0}}=\fgl_n$ for  $\fq_n,\fsq_n$
and  $\fg_{\ol{0}}=\fsl_n$ for $\mathfrak{psq}_n$, $\mathfrak{pq}_n$.
One has $\fsq_1=\mathbb{C}$ and $\mathfrak{psq}_1=0$; for all other cases
 $\fh\not=\ft$. In this case $\Delta=\Delta_{\ol{0}}=\Delta_{iso}$
(this set is empty if $n=1$).

\section{Grothendieck rings}\label{ringK}
In this section
we collect some general results about
the Grothendieck  rings $\cK(\fg)$ and the rings $\Ch(\fg)$, $\cA(\fg)$ in the case when $\fg$ is quasireductive. In~\Prop{lem-infty}
 we show that $\Ch(\fg)$ is a subring in $\cA(\fg)$.  

\subsection{Notation}\label{notat2}
If $\CC$ is a category of $\fg$-modules, the Grothendieck group $\cK(\CC)$
of $\CC$ is a free $\mathbb{Z}$-module spanned by $[N]$ with $N\in\CC$
subject to the relations $[M]+[N/M]=[N]$ for each pair
of  modules $M\subset N$;
we denote by $\cK_{\pm}(\CC)$ the quotient of $\cK(\CC)$ by the relations
$[N]=\pm [\Pi N]$. If $\CC$ is closed under the tensor product, $\cK(\CC)$ has a ring structure given by $[M]\cdot [N]:=[M\otimes N]$.

We denote by $\Fin(\fg)$ the full subcategory of finite-dimensional $\fg$-modules
and by $\cK(\fg)$ the Grothendieck ring of this category.

Recall that $\xi$ is a formal variable satisfying $\xi^2=1$. We identify
$\xi$ with the image of $\Pi\triv$ and denote by
 $\psi_{\pm}(\fg): \cK(\fg)\to\cK(\fg)/(\xi\mp 1)$  the canonical 
homomorphism. We set $\cK_{\pm}(\fg):=\psi_{\pm}(\fg)(\cK(\fg))$.

One readily sees that $\psi_+\times\psi_-$ gives an embedding
$\cK(\fg)\hookrightarrow \cK_+(\fg)\times \cK_-(\fg)$; this embedding is strict 
(for instance, $(0;1)$ does not belong to the image) and induces a bijection
$\cK(\fg)_{\mathbb{Q}}\iso 
\cK_+(\fg)_{\mathbb{Q}}\times \cK_-(\fg)_{\mathbb{Q}}$
(where for a $\mathbb{Z}$-module $M$ we set $M_{\mathbb{Q}}:=M\otimes_{\mathbb{Z}}\mathbb{Q}$).

\subsubsection{}\label{resnikov}
An exact functor $F:\CC\to\CC'$ induces a group homomorphism
$\cK(\fg)\to\cK(\fg')$ which is a ring homomorphism
if $F$ is a tensor functor. For the restriction functor
$\Res^{\fg}_{\fg'}$ we denote the corresponding homomorphism by $\res^{\fg}_{\fg'}$.
It is easy to check  that we have the following commutative diagrams

$$
\xymatrix{\cK(\fg) \ar[r]^{\res^{\fg}_{\fg'}} \ar[d]^{\psi_{\pm}(\fg)} & \cK(\fg')\ar[d]^{\psi_{\pm}} \\
 \cK_{\pm}(\fg)\ar[r]^{\res^{\fg}_{\fg'}}  & \cK_{\pm}(\fg') }
$$

\subsection{Structure of $\cK(\fg)$}\label{basisK}
We set 
$I_1:=\{\lambda\in\ft^*|\  \Pi C_{\lambda}\cong C_{\lambda}\}$.
We have
$$I_1=\{\lambda\in\ft^*|\  \Pi L(\lambda)\cong L(\lambda)\}.$$
Up to a parity shift $\Pi$
each  finite-dimensional simple module  is isomorphic to $L(\lambda)$
for some $\lambda\in P^+(\Delta^+)$. Thus
the ring $\cK(\fg)$ is a free $\mathbb{Z}$-module
with a basis 
$$\{[L(\lambda)]\}_{\lambda\in P^+(\Delta^+)}
\coprod\{\xi [L(\lambda)] \}_{\lambda\in P^+(\Delta^+)\setminus I_1}$$
where
 $\xi [L(\lambda)]=[L(\lambda)]$ for $\lambda\in P^+(\Delta^+)\cap I_1$.
The ring $\cK_+(\fg)$ is a free $\mathbb{Z}$-module
with a basis 
$\{\psi_+([L(\lambda)]\}_{\lambda\in P^+(\Delta^+)}$.

For  each $b=\displaystyle\sum_{\nu} m_{\nu} e^{\nu}$ with $m_{\nu}\in\mathbb{Z}[\xi]$
we set
$\displaystyle\supp(b):=\{\nu|\ m_{\nu}\not=0\}$.

\subsubsection{}
\begin{cor}{cor0res}
The maps $\res^{\fg}_{\fh}:\cK(\fg)\to\cK(\fh)$,  $\res^{\fg}_{\ft}:\cK_+(\fg)\to\cK_+(\ft)$
are injective.
\end{cor}
\begin{proof}
Take 
$a=\displaystyle\sum_{\lambda\in P^+(\Delta^+)} (k_{\lambda}+\xi n_{\lambda}) 
[L(\lambda)]\in\cK(\fg)$
with $k_{\lambda},n_{\lambda}\in\mathbb{Z}$ and $n_{\lambda}=0$ for $\lambda\in I_1$. It is easy to see that each 
 maximal element in the set 
$\{\lambda|\ k_{\lambda}+\xi n_{\lambda}\not=0\}$ 
 is a maximal element in
$\supp(\res^{\fg}_{\fh}(a))$; this implies the injectivity of  $\res^{\fg}_{\fh}:\cK(\fg)\to\cK(\fh)$.  The injectivity of the second map can be checked similarly.
\end{proof}

\subsection{Ring $\Ch(\fg)$}\label{RingCh}
Recall that  the rings $\Ch(\fg)$, $\Ch_{\pm}(\fg)$  are 
the images of the homomorphisms $\res^{\fg}_{\ft}:\cK(\fg)\to\cK(\ft)$
and $\res^{\fg}_{\ft}:\cK_{\pm}(\fg)\to\cK_{\pm}(\ft)$.
We set
$$\mathbb{Z}[P_0;\xi]:=\mathbb{Z}[P_0]\otimes_{\mathbb{Z}}\mathbb{Z}[\xi],\ \ \ \mathbb{Z}[\ft^*;\xi]:=\mathbb{Z}[\ft^*]\otimes_{\mathbb{Z}} \mathbb{Z}[\xi];$$
we identify $\cK(\ft)$ with
$\mathbb{Z}[\ft^*;\xi]$ and
$\cK_{\pm}(\ft)$ with $\mathbb{Z}[\ft^*]$.
We denote the homomorphisms $\psi_{\pm}(\ft)$ by $\psi_{\pm}$:
these maps are given by $\psi_{\pm}(\xi)=\pm 1$, $\psi_{\pm}(e^{\nu})=e^{\nu}$.

\subsubsection{}
Let $N$ be a finite-dimensional module and $[N]$ be its image in $\cK(\fg)$.
The map $\res^{\fg}_{\ft}:\cK(\fg)\to\cK(\ft)$ is given by 
$[N]\mapsto \ch_{\xi} N$, where
$$\ch_{\xi} N=\sum_{\nu} (\dim N_{\nu}\cap N_{\ol{0}}) e^{\nu}+
\xi \sum_{\nu} (\dim N_{\nu}\cap N_{\ol{1}}) e^{\nu}$$
and  $N_{\nu}$ is the (generalized) weight space of weight $\nu$. 
We set
$$\ch N:=\sum_{\nu\in\ft^*} \dim N_{\nu} e^{\nu}=\psi_+(\ch_{\xi} N)\ \ \ \sch N:=\sum_{\nu\in\ft^*} \sdim N_{\nu} e^{\nu}=\psi_-(\ch_{\xi} N)$$ 
and
 view  $\ch N,\sch N$ as  elements of $\mathbb{Z}[P_0]$.
By~\ref{resnikov} we have the following commutative diagrams 
$$
\xymatrix{\cK(\fg) \ar[r]^{\ch_{\xi}} \ar[d]^{\psi_{+}(\fg)}\ar[rd]^{\ch } & \Ch(\fg)\ar[d]^{\psi_{+}}  &  & & \cK(\fg) \ar[r]^{\ch_{\xi}} \ar[d]^{\psi_{-}(\fg)}\ar[rd]^{\sch } & \Ch(\fg)\ar[d]^{\psi_{-}} \\
 \cK_{+}(\fg)\ar[r] & \Ch_{+}(\fg) & & & 
           \cK_{-}(\fg)\ar[r] & \Ch_{-}(\fg) }
$$
All maps in the above diagrams are surjective.  

By~\Cor{cor0res},
one has $\cK_+(\fg)\iso \Ch_+(\fg)$ and $\cK(\fg)\iso\Ch(\fg)$
if $\fh=\ft$.
 In particular,
 $\Ch_+(\fg)$ (resp., $\Ch_-(\fg)$) 
is a $\mathbb{Z}$-span of $\ch N$ (resp., $\sch N$) with $N\in\Fin(\fg)$.

\subsubsection{Remark}
If $\fg'$ is a subalgebra of $\fg$ and
$\ft\subset\fg'$ we have 
$\res^{\fg'}_{\ft}\circ \res^{\fg}_{\fg'}=\res^{\fg}_{\ft}$ which gives
$\Ch(\fg)\subset\Ch(\fg')$.

\subsubsection{}
By~\Cor{cor0res} the ring $\Ch_+(\fg)$ is a free   $\mathbb{Z}$-module with a basis
$\ch L(\lambda)$ for $\lambda\in P^+(\Delta^+)$;
if $\fh=\ft$, then
 $\Ch_-(\fg)$ is a free   $\mathbb{Z}$-module with a basis
$\sch L(\lambda)$ for $\lambda\in P^+(\Delta^+)$.

In general, the ring  $\Ch_-(\fg)$ is spanned by 
$B_1:=\{\sch L(\lambda)|\ \lambda\in P^+(\Delta^+)\setminus I_1\}$  (since $\sch L(\lambda)=0$ if $\lambda\in I_1$). The agrument used in the proof
of~\Cor{cor0res}  shows
that the set 
$B_2:=\{\sch L(\lambda)|\ \lambda\in P^+(\Delta^+): \sdim C_{\lambda}\not=0\}$
is linearly independent. Clearly, $B_2\subset B_1$ and $B_1=B_2=\{\sch L(\lambda)|\ \lambda\in P^+(\Delta^+)\}$
if $\fh=\ft$.

\subsection{The ring $R(P_0)$}\label{RP0}
Recall that
$R(P_0):=\Ch(\fh)\cap \mathbb{Z}[P_0;\xi]$.
We set 
$$R_{\pm}(P_0):=\psi_{\pm}(R(P_0))=\Ch_{\pm}(\fh)\cap \mathbb{Z}[P_0].$$

Observe that $R(P_0)=\res^{\fh}_{\ft}\cK(\CC)$, where $\CC$ is the full
subcategory of the category $\Fin(\fh)$ consistsing of the modules
with the $\ft$-eigenvalues lying in $P_0$; since
$P_0$ is a subgroup of $\ft^*$, $\CC$ is closed under the tensor product,
so $R(P_0)$ (resp.,  $R_{\pm}(P_0)$) is a subring of $\mathbb{Z}[P_0;\xi]$
(resp., of $\mathbb{Z}[P_0]$). If $\fh=\ft$, then  $R(P_0)=\mathbb{Z}[P_0;\xi]$
and $R_{\pm}(P_0)=\mathbb{Z}[P_0]$.

\subsubsection{}\label{RP00}
Retain notation of~\ref{Clambda}.  
The ring $R(P_0)$ is a free $\mathbb{Z}$-module
with a basis 
$$\{\ch_{\xi} C_{\lambda}\}_{\lambda\in P_0}
\coprod\{\xi\ch_{\xi} C_{\lambda}\}_{\lambda\in P_0:\ \sdim C_{\lambda}\not=0}.$$
The rings $R_{\pm}(P_0)$ are free $\mathbb{Z}$-modules:
  $R_{+}(R_0)$ (resp., $R_-(P_0)$) consists 
 of the elements $\sum_{\nu\in P_0} m_{\nu} e^{\nu}$ with
$m_{\nu}$ divisible by $\dim C_{\nu}$ and $\sdim C_{\nu}$ respectively
(for $R_-(P_0)$ we have  $m_{\nu}=0$ if
 $\sdim C_{\nu}=0$).

\subsubsection{$W$-action}
The group $W$ acts on the ring $\mathbb{Z}[P_0;\xi]$ by $w(e^{\lambda})=e^{w\lambda}$ and $w\xi=\xi$.
Recall that each $w\in W$ 
acts on $\fh$ by an automorphism $\iota_w$;  the module $C_{w\lambda}$ 
is isomorphic to the $\iota_w^{-1}$-twist of $C_{\lambda}$, so
$\dim  C_{\lambda}=\dim  C_{w\lambda}$ and
$\sdim  C_{\lambda}=\sdim  C_{w\lambda}$. Therefore
$R(P_0), R_{\pm}(P_0)$ are $W$-stable.

\subsubsection{Remark}
For $\fg=\fq_n$ one has $\sdim C_{\lambda}=0$ for each $\lambda\not=0$;
this gives $\Ch_-(\fh)=\mathbb{Z}$.
Moreover, by~\cite{Cheng}, $\sch L(\lambda)=0$ for each $\lambda\not=0$.

\subsubsection{}
Since $\xi^2=1$ we have 
 the embedding 
$$\psi_+\times\psi_-: R(P_0)\to R_+(P_0)\times R_-(P_0).$$
We will use the following construction:
for any subsets  $A_{\pm}\subset R_{\pm}(P_0)$ we introduce
$$A_+\underset{R(P_0)}{\times} A_-:=\{a\in R(P_0)|\ \psi_{\pm}(a)\in A_{\pm}\}.$$
Note that  $A_+\underset{R(P_0)}{\times} A_-$ 
 is a subring of $R(P_0)$ if $A'_{\pm}$ are rings.
 
\subsubsection{}
\begin{lem}{lemunder}
Let $A$ be a $\mathbb{Z}[\xi]$-submodule of $R(P_0)$
with the following property: if $a\in R(P_0)$ and $2a\in A$, then $a\in A$.
Then
$$A=\psi_+(A)\underset{R(P_0)}{\times} \psi_-(A).$$
\end{lem}
\begin{proof}
Take 
$a\in R(P_0)$ such that $\psi_{\pm}(a)\in \psi_{\pm}(A)$. Then $A$ contains
$a-c(1-\xi)$  for some $c\in R(P_0)$.
Since $A$ is a  $\mathbb{Z}[\xi]$-submodule, $A$ contains 
$(1+\xi)(a-c(1-\xi))=(1+\xi)a$. Similarly, $A$ contains $(1-\xi)a$, so
 $2a\in A$. Then the assumption gives $a\in A$ as required.
\end{proof}

\subsection{Rings $\cA(\fg)$ and $\Ch(\fg)$}\label{RingsACh}
Recall that $\Ch_{\pm}(\fg)=\psi_{\pm}(\Ch(\fg))$ and that
$$\cA(\fg):=\bigcap_{\beta\in\Delta_{iso}}\{\sum_{\nu} m_{\nu} e^{\nu}\in R(P_0)^W\ |\   \langle\nu,h_{\beta}\rangle\not=0\ \Longrightarrow\ 
 \sum_{i\in\mathbb{Z}} 
(-\xi)^i  m_{\nu+i\beta}=0\}.$$
Setting $\cA_{\pm}(\fg):=\psi_{\pm}(\cA(\fg))$ we have 
\begin{equation}\label{Apm}
\cA_{\pm}(\fg):=\bigcap_{\beta\in\Delta_{iso}}\{\sum_{\nu} m_{\nu} e^{\nu}\in R_{\pm}(P_0)^W\ |\   \langle\nu,h_{\beta}\rangle\not=0\ \Longrightarrow\ 
 \sum_{i\in\mathbb{Z}} 
(\mp 1)^i  m_{\nu+i\beta}=0\}\end{equation}
(recall that  $R_{\pm}(P_0)=\mathbb{Z}[P_0]$ if $\fh=\ft$).

\subsubsection{}
\begin{cor}{corA+-Ch+-}
$\cA(\fg)=\cA_+(\fg)\underset{R(P_0)}{\times}\cA_-(\fg)\ $ and
$\ \Ch(\fg)=\Ch_+(\fg)\underset{R(P_0)}{\times}\Ch_-(\fg)$.
\end{cor}
\begin{proof}
For $\cA(\fg)$ the assumption of~\Lem{lemunder} hold.
The ring $\Ch(\fg)$ is a $\mathbb{Z}[\xi]$-submodule of $R(P_0)$.
It remains to verify that for
 $a\in R(P_0)$ with $2a\in \Ch(\fg)$ one has $a\in \Ch(\fg)$.
Write 
$2a=\sum_{\lambda\in P^+(\Delta^+)} m_{\lambda} \ch_{\xi} L(\lambda)$
where $m_{\lambda}\in \mathbb{Z}[\xi]$; without loss of generality
we may assume that
$m_{\lambda}\not\in 2\mathbb{Z}[\xi]$ for all $\lambda$. Let $\lambda$
be maximal such that $m_{\lambda}\not=0$. 
The coefficient of $e^{\lambda}$ in $2a$
is equal to $m_{\lambda}\ch_{\xi} C_{\lambda}$. Using~\ref{RP00}
we obtain $a\not\in R(P_0)$, a contradiction.
\end{proof}

\subsubsection{Remark}
The algebras $\cA_+(\fg)$ and $\cA_-(\fg)$ can be very different:
for example, for $\fq_n$, $\cA_-(\fg)=\mathbb{Z}$ whereas
$\cA_+(\fg)$ is a free $\mathbb{Z}$-module of infinite rank.
For Kac-Moody superalgebras these algebras are not so different, see~\ref{A+-iso} below.

\subsubsection{Remark}
If $\Delta_{iso}$ is empty, then $\cA(\fg)=R(P_0)^W$. If $\Delta_{iso}$ is non-empty and  the group generated by $W$ and $-\Id$ acts transitively on
$\Delta_{iso}$, then for any $\beta\in\Delta_{iso}$ we have
$$
\cA(\fg)=\{\sum_{\nu} m_{\nu} e^{\nu}\in R(\fg)^W\ |\   \langle\nu,h_{\beta}\rangle\not=0\ \Longrightarrow\ 
 \sum_{i\in\mathbb{Z}} 
(-\xi)^i m_{\nu+i\beta}=0\}.
$$
Note that the above condition
($W$ acts transitively on
$\Delta_{iso}/\{\pm 1\}$) holds for the Kac-Moody, $Q$-type and
$P$-type superalgebras.

%
%
%
%
%
%

\subsection{Case $\fgl(1|1)$}\label{gl11}
For the Lie superalgebra $\fgl(1|1)$ the algebra $\fg_{\ol{0}}$
is a two-dimensional commutative Lie algebra and
$\fg_{\ol{0}}=\fh=\ft$. One has $\Delta=\Delta_{iso}=\{\pm\beta\}$. 
Since $\beta\in\Delta_{iso}$ and $\dim \ft=2$ we have $\{\nu|\ \langle\nu,h_{\beta}\rangle=0\}=\mathbb{C}\beta$
A module $L(\lambda)$ is one-dimensional if $\lambda\in\mathbb{C}\beta$ and 
is   two-dimensional with $\ch_{\xi} L(\lambda)=e^{\lambda}(1+\xi e^{-\beta})$
if $\lambda\not\in\mathbb{C}\beta$.   Using~\ref{basisK} we obtain 
$$\Ch(\fg)=\{\sum m_{\nu}e^{\nu}|\ 
\sum_{i\in\mathbb{Z}} 
(-\xi)^i  m_{\nu+i\beta}=0\ \text{ for }\nu\not\in\mathbb{C}\beta\}=\cA(\fg).$$

\subsection{}
\begin{prop}{lem-infty}
For a quasireductive Lie superalgebra $\fg$ we have
 $\Ch(\fg)\subset \cA(\fg)$.
\end{prop}
\begin{proof}
Since $\fg_{\ol{0}}$ is a reductive Lie algebra one has 
$\Ch(\fg_{\ol{0}})=\mathbb{Z}[P_0;\xi]^W$.
The embeddings $\fg_{\ol{0}}\subset\fg$, $\fh\subset\fg$ give
$\Ch(\fg)\subset \Ch(\fg_{\ol{0}})$ and
 $\Ch(\fg)\subset \Ch(\fh)$; thus 
$$\Ch(\fg)\subset \Ch(\fg_{\ol{0}})\cap\Ch(\fh)=\mathbb{Z}[P_0;\xi]^W\cap \Ch(\fh)=R(P_0)^W.$$

Fix $\beta\in\Delta_{iso}$. By~\ref{setio}
the space $\fg_{\beta}+\ft+\fg_{-\beta}$ contains
a quasireductive subalgebra 
$$\fl:=\fgl(1|1)\times\ft'$$
and $\ft\subset\fl$ is the Cartan subalgebra of $\fl$.
View $\Ch(\fl)$
as a subring in $\mathbb{Z}[\ft^*;\xi]$. By~\ref{gl11}
 $\sum m_{\nu} e^{\nu}\in \Ch(\fl)$ implies $\sum_{\in\mathbb{Z}}  (-\xi)^i  m_{\nu+i\beta}=0$ if 
$\langle\nu,h_{\beta}\rangle\not=0$.
Since 
$\Ch(\fg)\subset  \Ch(\fl)$ we obtain  $\Ch(\fg)\subset \cA(\fg)$.
\end{proof}

\section{Short basis}\label{sect1}
In this section we show that  $\cA(\fg)=\Ch(\fg)$ and this ring admits
a ``short basis''  if $\fg$ is a 
Kac-Moody superalgebras $\fg\not=\fgl(1|1)$  or a $Q$-type
superalgebra (see Corollaries~\ref{thmSV},\ref{corglnn}).
Recall that $\cA(\fgl(1|1))=\Ch(\fgl(1|1))$; it is easy to see
that this ring does not admit a ``short basis''.

\subsection{Definition}\label{short}
We call a  $\mathbb{Z}$-basis $B\coprod \xi B'$ of $\Ch(\fg)$ {\em short with respect to $\Delta^+$}
if 
\begin{equation}\label{eqshort}\begin{array}{lll}
\text{(a) }& & B:=\{b_{\lambda}|\ \lambda\in P^+(\Delta^+)\},\ \ \ \ 
B':=\{b_{\lambda}|\ \sdim C_{\lambda}\not=0\};\\
\text{(b)} & &
\supp( b_{\lambda}-\ch_{\xi} C_{\lambda})\cap  P^+(\Delta^+)=\emptyset;\\
\text{(c) }& & \xi b_{\lambda}=b_{\lambda}\ \text{  if } \sdim C_{\lambda}=0;\\
\text{(d)} & & \supp(b_{\lambda})\subset\{\nu\in\ft^*|\ \nu\leq\lambda\}.\end{array}
\end{equation}

\subsection{Properties of short bases}
Assume that $\Ch(\fg)$ admits a short basis.
\subsubsection{}\label{01}
It is not hard to show (see, for example,~\cite{GSS}):
that 
$\dim C_{\lambda}=1$ or  $\sdim  C_{\lambda}=0$. Therefore 
$B'=\{b_{\lambda}|\ \dim C_{\lambda}=1\}$.

\subsubsection{}\label{mla}
Since $\ch_{\xi} L(\lambda)\in \Ch(\fh)\cap \mathbb{Z}[P_0;\xi]=R[P_0]$, for
 each $\nu$ we have
$\ch_{\xi} L(\lambda)_{\nu}=m_{\lambda,\nu}\ch_{\xi} C_{\nu}$
for some $m_{\lambda,\nu}\in\mathbb{Z}[\xi]$ with $m_{\lambda,\nu}\in\mathbb{Z}$
if $\sdim C_{\nu}=0$.
 The property (b) gives
$$\ch_{\xi} L(\lambda)=
b_{\lambda}+\sum_{\nu\in  P^+(\Delta^+)}m_{\lambda,\nu}b_{\nu}.$$

Set  $b^{\pm}_{\nu}:=\psi_{\pm}(b_{\nu})$; then $b^{\pm}_{\nu}\in\mathbb{Z}[P_0]$.
By (b), (c) we get $b^-_{\nu}\not=0$ if and only if
$b_{\nu}\in B'$. Moreover, the sets $\{\psi_+(b)|\ b\in B\}$  
and $\{\psi_-(b)|\ b\in B'\}$
form $\mathbb{Z}$-bases of   $\Ch_{\pm}(\fg)$ and
$$\ch L(\lambda)=\sum \frac{\dim L(\lambda)_{\nu}}{\dim C_{\nu}}\  b^+_{\nu},\ \ \ \ \ 
\sch L(\lambda)= \sum_{\nu: \dim C_{\nu}=1} \sdim L(\lambda)_{\nu}\  b^-_{\nu}.$$
(for the second formula we used 
$\sdim  C_{\lambda}\in\{0,1\}$).

\subsection{Set $Y_{\lambda}$}\label{exitsshort}
For each $\lambda$ we set
$$Y_{\lambda}:=\{\mu\in P^+(\pi)|\ \mu<\lambda\}.$$
If the sets $Y_{\lambda}$ are finite, the 
elements
 $b_{\lambda}$ satisfying the property (b) 
can be constructed by the induction on the cardinality of $Y_{\lambda}$ via the formula
$b_{\lambda}:=\ch L(\lambda)-\sum_{\nu\in Y_{\lambda}\cap P^+(\Delta^+)} 
m_{\lambda,\nu} b_{\nu}$ where $m_{\lambda,\nu}$ are as in~\ref{mla}. 
Note that $b_{\lambda}$ satisfies
the property (d).

\subsubsection{}
\begin{lem}{lemele1}
If
$(-\mathbb{R}_{\geq 0}\Delta^+)\cap \mathbb{R}_{\geq 0}P^+(\pi)=\{0\}$, then
 $Y_{\lambda}$ is finite
 for each $\lambda$.\end{lem}
\begin{proof}
View $\lambda-\mathbb{R}\Delta^+$ as an affine space. Let us show that the set
$$M:=(\lambda-\mathbb{R}_{\geq 0}\Delta^+)\cap \mathbb{R}_{\geq 0}P^+(\pi)$$
is bounded. Both $\lambda-\mathbb{R}_{\geq 0}\Delta^+$ and $ 
\mathbb{R}_{\geq 0}P^+(\pi)$ are polyhedra (i.e.,  
the intersection of finitely many half-spaces), so $M$ is 
a polyhedron as well. Assume that $M$ is unbounded. Then $M$ contains 
a ray, that is for some $\nu\in M$ and a non-zero $\gamma\in\mathbb{R}\Delta^+$
one has $\nu-x\gamma\in M$ for all $x\geq 0$.
Since $\nu-x\gamma\subset \lambda-\mathbb{R}_{\geq 0}\Delta^+$ we get
$$\gamma\in \frac{\nu-\lambda}{x}+\mathbb{R}_{\geq 0}\Delta^+\ \ \text{ for all } x>0.$$
Since $\mathbb{R}_{\geq 0}\Delta^+$ is closed, we obtain 
$\gamma\in \mathbb{R}_{\geq 0}\Delta^+$.
Similarly, $\nu-x\gamma\in\mathbb{R}_{\geq 0}P^+(\pi)$ implies
$\gamma\in -\mathbb{R}_{\geq 0}P^+(\pi)$.  Thus
$\gamma\in -\mathbb{R}_{\geq 0}P^+(\pi)\cap \mathbb{R}_{\geq 0}\Delta^+$
and the assumption gives $\gamma=0$, a contradiction.
Hence $M$ is bounded.

Let $h\in\ft$ be a ``defining element for $\Delta^+$'', i.e.
$\Ree \langle\alpha,h\rangle\not= 0$ for all 
$\alpha\in\Delta$ and
$\alpha\in\Delta^+$ if and only if $\Ree\langle\alpha,h\rangle>0$.
Set 
$$a_-:=\min_{\mu\in M}  \Ree\langle \mu,h\rangle,\ \ 
 \  a_+:=\max_{\mu\in M} \Ree\langle \mu,h\rangle,\ \ \ 
u:=\min_{\alpha\in\Delta^+} \Ree\langle\alpha,h\rangle.$$
Any $\mu\in Y_{\lambda}$ is of the form  
$\mu=\sum_{\alpha\in\Delta^+} \lambda-k_{\alpha}\alpha$  for some $k_{\alpha}\in\mathbb{N}$.
Since $\mu\in Y_{\lambda}\subset M$ one has
$a_-\leq u\sum_{\alpha\in\Delta^+}  k_{\alpha}\leq a_+$.
Thus each $k_{\alpha}$ has finitely many possible values, so
 $Y_{\lambda}$ is finite.
\end{proof}

\subsubsection{Remark}\label{remele}
Since $\fg_{\ol{0}}$ is reductive, one has $-w_0\pi=\pi$
for some $w_0\in W$. We claim that 
\begin{equation}\label{remele}
-w_0\Delta^+=\Delta^+\ \ \Longrightarrow\ \ \ (-\mathbb{R}_{\geq 0}\Delta^+)\cap \mathbb{R}_{\geq 0}P^+(\pi)=\{0\}.\end{equation}
Indeed, assume that  $-w_0\Delta^+=\Delta^+$. If $\nu$ is a non-zero element of
 $(-\mathbb{R}_{\geq 0}\Delta^+)$, then
 $w_0\nu\in \mathbb{R}_{\geq 0}\Delta^+$ and so $w_0\nu-\nu$
 is a non-zero element in $\mathbb{R}_{\geq 0}\Delta^+$. By~\Lem{lem-1} we obtain
  $\nu\not\in\mathbb{R}_{\geq 0}P^+(\pi)$.

\subsection{}
\begin{lem}{cor0}
Assume  that 
\begin{itemize}

\item $ \mathbb{R}_{\geq 0}P^+(\pi)\cap (-\mathbb{R}_{\geq 0}\Delta^+)=\{0\}$;

\item
for any non-zero $y\in \cA(\fg)$ one has $\supp(y)\cap P^+(\Delta^+)\not=\emptyset$.
\end{itemize} 
Then $\cA(\fg)=\Ch(\fg)$ and $\Ch(\fg)$  admits a unique short $\mathbb{Z}$-basis
with respect to $\Delta^+$.
\end{lem}
\begin{proof}
By~\Lem{lemele1}  the sets 
$Y_{\lambda}$ are finite. By~\ref{exitsshort} the elements
 $b_{\lambda}$ satisfying the properties (b), (d) of~(\ref{eqshort})
can be uniquely defined.
Take $\lambda$ such that $ \sdim C_{\lambda}=0$ and set
$y:=\xi b_{\lambda}-b_{\lambda}$. One has
$$\supp(y)\cap  P^+(\Delta^+)\subset \supp (b_{\lambda})\cap  P^+(\Delta^+)=\{\lambda\}$$
and $\lambda\not\in \supp(y)$ (since
 $\sdim C_{\lambda}=0$). Hence $\supp(y)\cap  P^+(\Delta^+)=\emptyset$.
 Since $y\in\Ch(\fg)$ and $\Ch(\fg)\subset\cA(\fg)$, the second assumption
 gives $y=0$, that is $\xi b_{\lambda}=b_{\lambda}$. This gives the property 
(c) of~(\ref{eqshort}).
Hence the elements $b_{\lambda}$ satisfies the properties (b), (c), (d) 
 of~(\ref{eqshort}). Let us show that $B\coprod \xi B'$
forms a $\mathbb{Z}$-basis of $\cA(\fg)$. 

Let us verify that the elements of $B\coprod \xi B'$ are linearly 
independent over $\mathbb{Z}$. Let
$$\sum_{\lambda\in P^+(\Delta^+)} n_{\lambda} b_{\lambda}=0$$
for some $n_{\lambda}\in\mathbb{Z}[\xi]$ with $n_{\lambda}\in\mathbb{Z}$
if $\sdim C_{\lambda}=0$.
Observe that for each $\lambda\in P^+(\Delta^+)$ the coefficient of $e^{\lambda}$
in the above expression is 
 $n_{\lambda}\ch_{\xi} C_{\lambda}$, so 
 $n_{\lambda}\ch_{\xi} C_{\lambda}=0$. If $\sdim C_{\lambda}=0$, 
this gives $n_{\lambda}=0$ (because $n_{\lambda}\in\mathbb{Z}$).
 If $\sdim C_{\lambda}\not=0$, then
the formulae $\psi_{\pm}(n_{\lambda}\ch_{\xi} C_{\lambda})=0$
give $\psi_{\pm}(n_{\lambda})=0$, so $n_{\lambda}=0$.
Hence $B\coprod \xi B'$ are linearly independent.

Let us check that $\cA(\fg)$ lies in $\mathbb{Z}[\xi]$-span of  $B$.
Take $a\in\cA(\fg)$.
Since $\cA(\fg)\subset R(P_0)$ we have
$a=\sum_{\nu} n_{\nu}\ch_{\xi} C_{\nu}$
for some $n_{\nu}\in\mathbb{Z}[\xi]$. Since $b_{\lambda}\in \Ch(\fg)\subset\cA(\fg)$ 
we obtain
$$a':=a-\sum_{\lambda\in P^+(\Delta^+)}n_{\lambda}b_{\lambda} \in\cA(\fg).$$
By above, $\supp(a')\cap P^+(\Delta^+)=\emptyset$, so the second assumption gives $a'=0$. Therefore $\cA(\fg)$ lies in $\mathbb{Z}[\xi]$-span of  $B$. 

By above,
$B\coprod \xi B'$ is a $\mathbb{Z}$-basis of $\cA(\fg)$. 
 Since $B\subset\Ch(\fg)$
and $\Ch(\fg)\subset\cA(\fg)$, we get $\cA(\fg)=\Ch(\fg)$.
\end{proof}

\subsection{}
\begin{lem}{lem0}
Assume that $\Delta^+$ admits a base $\Sigma$ and that
\begin{equation}\label{P+nice}
P^+(\Delta^+)\subset \{\lambda\in P^+(\pi)|\ \forall \beta\in\Sigma\cap\Delta_{iso}\ \
\langle\lambda,h_{\beta}\rangle\not=0\ \Longrightarrow \lambda-\beta\in P^+(\pi)\}.\end{equation}
Then
for any non-zero $y\in \cA(\fg)$ one has $\max\supp(y)\subset P^+(\Delta^+)$.
\end{lem}
\begin{proof}
Set $S:=\supp(y)$.  Since $y\in\cA(\fg)$ the set $S$ has the following properties:
\begin{enumerate}
\item $WS=S$ and  $S\subset P_0$;
\item  
if $\langle \eta,h_{\beta}\rangle\not=0$ for some $\beta\in\Delta_{iso}$,
then  $S\cap (\eta+\mathbb{Z}\beta)\not=\{\eta\}$.
\end{enumerate}

Let $\eta$ be a maximal element in $\supp(y)$. Assume that $\eta\not\in P^+(\Delta^+)$. Combining (i) and~\Lem{lem-1} we have   $\eta\in P^+(\pi)$, so 
 $\eta\in P^+(\pi)\setminus P^+(\Delta^+)$. By~(\ref{P+nice})
there exists $\beta\in\Sigma\cap\Delta_{iso}$ such that
$$\langle\eta, h_{\beta}\rangle \not=0,\ \  \eta-\beta\not\in P^+(\pi).$$
 By  (ii) $\eta-j\beta\in S$
for some non-zero integral $j$;
the maximality of $\eta$ gives $j\geq 1$.

Since $\eta\in P^+(\pi)$ and
$\eta-\beta\not\in P^+(\pi)$ there exists $\alpha\in\pi$ such that 
$s_{\alpha}\eta\leq \eta$ and $s_{\alpha}(\eta-\beta)>\eta-\beta$.
In particular, $s_{\alpha}\beta<\beta$. 
Viewing $\fg$ as an 
 $\fsl_2$-module for the 
$\fsl_2\subset\fg_{\ol{0}}$  corresponding to $\alpha$, we get
$\beta-\alpha\in \Delta$. Since $\beta$ is a simple root, both
$s_{\alpha}\beta$ and $\beta-\alpha$ are negative roots.
 Since $s_{\alpha}(\eta-\beta)=\eta-\beta+i\alpha$
for some $i\in\mathbb{Z}_{>0}$, this gives
$s_{\alpha}(\eta-\beta)>\eta$.
Combining with $-s_{\alpha}\beta>0$ we obtain
$s_{\alpha}(\eta-j\beta)>\eta$.
Since $\eta-j\beta\in S$, we have $s_{\alpha}(\eta-j\beta)\in S$
 which contradicts to the maximality of $\eta$.
\end{proof}

%

\subsection{}
\begin{cor}{cormainstep0}
Let $\fg$ be a quasireductive Lie superalgebra with a base $\Sigma$ satisfying~(\ref{P+nice}) and such that $
-\mathbb{R}_{\geq 0}\Delta^+\cap \mathbb{R}_{\geq 0}P^+(\pi)=\{0\}$.
Then $\Ch(\fg)= \cA(\fg)$ and this ring admits a short basis with respect to $\Delta^+$.
\end{cor}
\begin{proof}The assertion follows from Lemmatta~\ref{cor0} and~\ref{lem0}.
\end{proof}

\subsection{}
\begin{cor}{thmSV} Let $\fg$ be a $Q$-type superalgebra or
a finite-dimensional Kac-Moody superalgebra and $\fg\not=\fgl(n|n)$. Then
 $\Ch(\fg)=\cA(\fg)$ and this  ring
  admits a short  basis
with respect to the mixed triangular decomposition.
\end{cor}
\begin{proof}
For $Q$-type case one has $\Delta^+=\Delta^+_{\ol{0}}=\Delta_{iso}$.
In particular, $\Delta^+$ and $-\Delta^+$ are $W$-conjugated.
By~\cite{Pe},
$$P^+(\Delta^+)=\{\lambda\in P^+(\pi)|\ \forall \alpha\in \pi\ \ 
s_{\alpha}\lambda=\lambda\ \ \Longrightarrow\ \ 
\langle\lambda,h_{\alpha}\rangle=0\}.$$
Take $\eta\in P^+(\pi)\setminus P^+(\Delta^+)$. By above,
there  exists $\alpha\in \pi$ such that $\langle\eta,h_{\alpha}\rangle\not=0$ and
$s_{\alpha}\eta=\eta$. Since $s_{\alpha}(\eta-\alpha)=\eta+\alpha>\eta-\alpha$
we have $\eta-\alpha\not\in P^+(\pi)$. Therefore~(\ref{P+nice})
holds. Combining~(\ref{remele}) and~\Cor{cormainstep0} we obtain
the assertion for $Q$-type superalgebras.

Now let  $\fg$ be one of the algebras $\fgl(m|n)$ for $m\not=n$,
$\mathfrak{osp}\left(m|n\right)$,  $D\left(2|1;a\right)$, $F(4)$ or $G(3)$.
The mixed bases for these superalgebras
satisfy~(\ref{P+nice}), see Appendix.
It remains  to verify the 
formula
$(-\mathbb{R}_{\geq 0}\Delta^+)\cap \mathbb{R}_{\geq 0}P^+(\pi)=\{0\}$.
For  $D\left(2|1;a\right)$, $F(4)$, $G(3)$ and $\osp(2m+1|n)$  the Weyl group $W$ contains $-\Id$;
for $\osp(2m|2n)$ the sets $-\Delta^+$ and $\Delta^+$ are $W$-conjugated
if $\Sigma$ is a mixed base, see Appendix. Thus~(\ref{remele}) implies
the required formula for these cases. For the remaining case $\fgl(m|n)$
with $m\not=n$ the formula is verified in~\Lem{lemglmn} below.
\end{proof}

\subsection{Case $\fgl(m|n),\fsl(m|n)$}\label{glmn}
One has $\fgl(m|n)_{\ol{0}}=\fgl_m\times\fgl_n$. 
Let $\vareps_1,\ldots,\vareps_m,\delta_1,\ldots,\delta_n$ be the weights
of the standard representation of $\fgl(m|n)$ and 
$$\pi=\{\vareps_i-\vareps_{i+1}\}_{i=1}^{m-1}\cup
\{\delta_j-\delta_{j+1}\}_{j=1}^n.$$

Take $\fg=\fgl(m|n)$ or $\fg=\fsl(m|n)$ for $m\not=n$.
The bases compatible with the triangular decomposition
of $\fg_{\ol{0}}$ can be encoded by words in $m$ letters
$\vareps$ and $n$ letters $\delta$: for instance, the word $\vareps^2\delta$
correspond to the distinguished base $\{\vareps_1-\vareps_2,\vareps_2-\delta_1\}$ and
 the word $(\vareps\delta)^n\vareps^{m-n}$  corresponds to the mixed base 
$$
\Sigma_{m|n}:=\{\vareps_1-\delta_1,\delta_1-
\vareps_2,\ldots,\vareps_n-\delta_n,\delta_n-\vareps_{n+1},
\vareps_{n+1}-\vareps_{n+2},\ldots,
\vareps_{m-1}-\vareps_m\}$$
The distinguished bases
correspond to the words $\vareps^m\delta^n$ and $\delta^n\vareps^m$.
The bases $\Sigma$ and  $-w_0\Sigma$ correspond to the inverse words:
for example, $-w_0\Sigma_{m|n}$ corresponds to $\vareps^{m-n}(\delta\vareps)^n$.

\subsubsection{}
\begin{lem}{lemglmn}
Take $\fg=\fgl(m|n)$ or $\fg=\fsl(m|n)$ for $m\not=n$.
If the first and the last letter in the word corresponding to $\Sigma$
are the same, then  $(-\mathbb{R}_{\geq 0}\Delta^+)\cap \mathbb{R}_{\geq 0}P^+(\pi)=\{0\}$.
\end{lem}
\begin{proof}
Take $\nu\in (-\mathbb{R}_{\geq 0}\Delta^+)\cap \mathbb{R}_{\geq 0}P^+(\pi)$
and write
$\nu=\sum x_i\vareps_i+\sum y_i\delta_i$
for $x_i,y_i\in\mathbb{R}$. 
The condition $\nu\in  \mathbb{R}_{\geq 0}P^+(\pi)$
gives 
$$x_1\geq x_2\ldots\geq x_m,\  \ \ y_1\geq y_2 \ldots\geq y_n.$$

Let  the first and the last letter in the word corresponding to $\Sigma$
be $\vareps$.  Then the condition $\nu \in-\mathbb{R}_{\geq 0}\Delta^+$ gives
 $x_1\leq 0\leq x_m$, so $x_i=0$ for all $i$. Now
the condition $\nu \in-\mathbb{R}_{\geq 0}\Delta^+$ gives $y_1\leq 0\leq y_m$, so 
$y_j=0$ for all $j$.
\end{proof}

\subsubsection{Case $\fgl(n+1|n+1)$}\label{Sigmann}
We fix the base corresponding to $(\vareps\delta)^n\delta\vareps$:
$$\Sigma=\{\vareps_1-\delta_1,\delta_1-
\vareps_2,\ldots,\vareps_n-\delta_n,
 \delta_n-\delta_{n+1},\delta_{n+1}-\vareps_{n+1}\}.$$
 We will show that $\Ch(\fg)=\cA(\fg)$ and $\Ch(\fg)$ admits a short basis with respect
 to $\Delta^+(\Sigma)$.

\subsubsection{}
\begin{lem}{lemglnn0}
For any non-zero $y\in \cA(\fg)$ one has $\supp(y)\cap P^+(\Delta^+)\not=\emptyset$.
\end{lem}
\begin{proof}
We retain notation of~\ref{KM}.
We denote by $(-|-)$ a standard non-degenerate bilinear form on $\ft^*$;
for $\beta\in\Delta_{iso}$ one has
$(\mu|\beta)=0$ if and only if $\langle\mu,h_{\beta}\rangle=0$.

Set $S:=\supp(y)$. Recall that 
\begin{enumerate}
\item $WS=S$ and  $S\subset P_0$;
\item  
if $(\eta,\beta)\not=0$ for some $\beta\in\Delta_{iso}$,
then $S\cap (\eta+\mathbb{Z}\beta)\not=\{\eta\}$.
\end{enumerate}

We fix a total order $\succ$ on $\Sigma$:
$\Sigma=\{\alpha_1\succ\alpha_2\succ \ldots\succ\alpha_{\ell}\}$
and extend it to the lexicographic order
on $\mathbb{R}\Delta$ by setting $\sum k_i\alpha_i>0$
if $k_1>0$ or $k_1=0$, $k_2>0$ and so on. For $\lambda,\nu\in P_0$ we set
 $\lambda\succ\nu$ if $\lambda-\nu\succ 0$.
 Note that 
 $$\lambda>\nu\ \ \Longrightarrow\ \ \lambda\succ\nu.$$

Let $\eta$ be a maximal element in $\supp(y)$ with respect to the order
$\succ$.  

Combining (i) and~\Lem{lem-1} we have   $\eta\in P^+(\pi)$.
Assume that $\eta\not\in P^+(\Delta^+)$. Set $L:=L(\eta)$.
By~\Prop{propi} there exists 
$\alpha\in\pi$ such that for any $\Sigma'$ containing $\alpha$ one has
$\langle\hwt_{\Sigma'} L,\alpha^{\vee}\rangle\not\in\mathbb{N}$.
Fix such $\alpha$.
Note that $\langle\eta,\alpha^{\vee}\rangle\in\mathbb{N}$ since $\eta\in P^+(\pi)$. In particular, $\hwt_{\Sigma'} L\not=\eta$ and $\alpha\not\in\Sigma$.
One readily sees that for $\alpha\in\pi\setminus \Sigma$ one has 
$\alpha\in\Sigma'$, where either
$$\Sigma'=r_{\beta}\Sigma \ \text{ for }\beta\in\Sigma_{iso},\ s_{\alpha}\beta<0$$
or, for $\alpha=\vareps_n-\vareps_{n+1}$, one has
$$\Sigma'= r_{\beta'}r_{\beta}\Sigma=r_{\beta}r_{\beta'}\Sigma\ \text{ for }
\beta,\beta'\in\Sigma_{iso},\ \ (\beta|\beta')=0,\ \ 
\ \ s_{\alpha}\beta,s_{\alpha}\beta'<0.$$
(The property $(\beta|\beta')=0$ gives 
$\beta'\in  r_{\beta}\Sigma_{iso}$, so 
$r_{\beta'}r_{\beta}\Sigma$ is well-defined).
Set 
$$\gamma:=\eta-\hwt_{\Sigma'} L.$$
By above, $\gamma\not=0$, so 
$\gamma=\beta$ if  $\Sigma'=r_{\beta}\Sigma$ and 
$\gamma\in\{\beta,\beta',\beta+\beta'\}$
if $\Sigma'= r_{\beta'}r_{\beta}\Sigma$. Observe that 
$\gamma<\alpha$ in all cases.
Combining $\eta\in P^+(\pi)$ and
$\langle\eta-\gamma,\alpha^{\vee}\rangle\not\in\mathbb{N}$ we get
$$s_{\alpha}(\eta-\gamma)-(\eta-\gamma)=(k+1)\alpha\ \ \text{ for } k\in\mathbb{N}.$$
Therefore
$s_{\alpha}(\eta-\gamma)-(\eta-\gamma)\geq\alpha>\gamma$
which implies
$$s_{\alpha}(\eta-\gamma)>\eta.$$

Recall that $-s_{\alpha}\beta>0$ and $-s_{\alpha}\beta'>0$ if
$\Sigma'= r_{\beta'}r_{\beta}\Sigma$.

If $\gamma=\beta$, then $(\eta|\beta)\not=0$ and 
the property (ii) gives
$\eta-j\beta\in S$ for some $j\geq 1$. Then $S$ contains
$$s_{\alpha}(\eta-j\beta)=s_{\alpha}(\eta-\gamma)-(j-1)s_{\alpha}\beta>\eta,$$
a contradiction. The case $\gamma=\beta'$ is similar. Consider the remaining case
$\gamma=\beta+\beta'$. Then $(\eta|\beta),(\eta|\beta')\not=0$.
Let $\beta'\succ\beta$. By (ii) $\eta-j\beta\in S$ for some $j\in\mathbb{N}$
and then $\eta-j\beta-i\beta'\in S$ for some non-zero integral $i$.
If $i<0$, then $\eta-j\beta-i\beta'\succ\eta$ which contradicts
to the maximality of $\eta$. Hence $i\geq 1$.
Then
$$s_{\alpha}(\eta-j\beta-i\beta')=s_{\alpha}(\eta-\gamma)-(j-1)s_{\alpha}\beta
-(i-1)s_{\alpha}\beta'>\eta,$$
a contradiction. Hence $\eta\in P^+(\Delta^+)$ as required.
\end{proof}

Combing Lemmatta~\ref{cor0}, \ref{lemglmn}, \ref{lemglnn0} we obtain the

\subsubsection{}
\begin{cor}{corglnn}
For $\fgl(n|n)$ one
has $\Ch(\fg)=\cA(\fg)$; this ring admits a short basis with respect
 to $\Delta^+(\Sigma)$.
\end{cor}

\section{Grothendieck rings for subcategories of $\Fin(\fg)$}\label{Reform}
In this section we show that our description of $\cA_{-}(\fg)$ is equivalent 
to Sergeev-Veselov description. In~\ref{pn} we deduce the formula $\Ch(\fp_n)=\cA(\fp_n)$
from the results of~\cite{IRS}.

\subsection{Category $\cF(P)$}\label{FP}
The group $\mathbb{Z}\Delta$
acts on $P_0$ by the shifts $t_{\gamma}(\mu):=\mu+\gamma$ for $\gamma\in  \mathbb{Z}\Delta$ and $\mu\in P_0$. Let
 $P\subset P_0$ be an invariant subset (i.e., $\mu+\alpha\in P$ 
for all $\mu\in P$ and $\alpha\in \Delta$).
We denote by $\cF(P)$ (resp., $\cF_{\fh}(P)$) the full subcategory
of $\Fin(\fg)$ (resp., of $\Fin(\fh)$) with the weights
in $P$.

\subsubsection{}
We set $R(P):=\res^{\fh}_{\ft}\cK(\cF_{\fh}(P))$.
In other words, $R(P)$ 
is a $\mathbb{Z}[\xi]$-span of $\ch_{\xi} C_{\nu}$ for $\nu\in P$. 
The ring $R(P_0)$ corresponds to the case $P=P_0$.
Note that $\ch_{\xi} L(\lambda)\in R(P)$ for each $\lambda\in P$.
We denote by $\Ch(\cF(P))$ the image $\res^{\fg}_{\ft}(\cK(\cF(P))$.

Since $P+\mathbb{Z}\Delta\subset P$, we have $WP=P$ and so 
$W(R(P))=R(P)$.
We set
$$\cA(P):=\cA(\fg)\cap R(P).$$

\subsubsection{}
\begin{lem}{lemchP}
One has $\Ch(\cF(P))=R(P)\cap \Ch(\fg)$.
In particular,
$\Ch(\fg)=\cA(\fg)$ implies $\Ch(\cF(P))=\cA(P)$ 
for any $P$.
\end{lem}
\begin{proof}
The inclusion $\subset$ is straightforward. For the inverse inclusion 
take $a\in R(P)\cap \Ch(\fg)$. 
Write $a=\sum_{i=1}^s m_i\ch_{\xi} L(\lambda_i)$  where 
$m_i\in\mathbb{Z}[\xi]$.
 If $\lambda_s$ is maximal
among $\lambda_1,\ldots,\lambda_s$, then
$\lambda_s\in\supp(a)$, so $\lambda_s\in P$ and
$\ch_{\xi} L(\lambda_s)\subset R(P)$. Hence
 $a':=\sum_{i=1}^{s-1}  m_i\ch_{\xi} L(\lambda_i)$ lies in  $R(P)\cap \Ch(\fg)$.
Using the induction on $s$ we get $a'\in \Ch(\cF(P))$, so $a\in \Ch(\cF(P))$ as required.  
\end{proof}

\subsubsection{Remark}
Using the same agrument for $\ch L(\lambda_i)$ we obtain
$$\Ch_{+}(\cF(P))=R(P)\cap \Ch_{+}(\fg).$$ 
For $\Ch_-$ a similar argument works if 
$\fh=\ft$  or if $\cA_-(\fg)$ admits a short basis defined in~\ref{short}.

\subsubsection{Examples}\label{Pint}
Let $\fg$ be one of the algebras $\fgl(m|n),\osp(m|n),\fq_n$.
Let $P_{int}$ be the integral lattice which is the lattice spanned by the weights
of the standard representation. By~\Lem{lemchP} the formula 
$\Ch(\fg)=\cA(\fg)$ implies $\Ch(\cF(P_{int}))=R(P_{int})\cap\cA(\fg)$.

Let $\fg:=\fq_n$ and $P$ be the set of ``half-integral weights''. The formula $\Ch(\fg)=\cA(\fg)$ implies $\Ch(\cF(P))=R(P)\cap\cA(\fg)$.

\subsubsection{}
Similarly to~\Cor{corA+-Ch+-} we have
\begin{equation}\label{corAp+-}
\cA(P)=\cA_+(P)\underset{R(P_0)}{\times}\cA_-(P)\ \ \ \ 
 \Ch(P)=\Ch_+(P)\underset{R(P_0)}{\times}\Ch_-(P).\end{equation}

\subsubsection{}\label{A+-iso}
Let $\fg$ be such that $\fh=\ft$ and that
for some group homomorphism $p:\mathbb{Z}\Delta\to\mathbb{Z}_2$ 
one has $p(\Delta_{\ol{i}})=\ol{i}$. Assume, in addition, that
$P$ is a subgroup of $P_0$ which contains $\mathbb{Z}\Delta$ and  that
$p$ can be extended to the  
group homomorphism $p:\mathbb{Z}\Delta\to\mathbb{Z}_2$.
In this case $\mathbb{Z}[P]=R_{\pm}(P)$ admits an automorphism 
$\iota: e^{\nu}\mapsto (-1)^{p(\nu)} e^{\nu}$. We claim that 
this automorphism induces ring isomorphisms
$$\cA_-(P)\iso \cA_+(P),\ \ \Ch_-(P)\iso\Ch_+(P).$$
Indeed, since $\iota$ is an  automorphism, it is enough to check that
$\iota(\cA_-(P))=\cA_+(P)$ and $\iota(\Ch_-(P))=\Ch_+(P)$.
The first formula immediately follows from~(\ref{Apm}).
For the second formula note that the condition $\fh=\ft$ implies that
 $\Ch_{\pm}(P)$
are free $\mathbb{Z}$-modules with  bases 
$\{\ch L(\lambda)\}_{\lambda\in P\cap P^+(\Delta^+)}$ 
and $\{\sch L(\lambda)\}_{\lambda\in P\cap P^+(\Delta^+)}$ respectively.
One readily sees $\iota(\ch L(\lambda))=\sch L(\lambda)$;
this gives the required formula.

\subsubsection{}
The above assumptions hold, for example, if $\fg$ is a finite-dimensional 
Kac-Moody  superalgebra and $P=\mathbb{Z}\Delta$ or if
$\fg=\fgl(m|n),\osp(m|n),\fp_n$ and $P=P_{int}$ (see~\ref{Pint}).

\subsection{Sergeev-Veselov description  of $\cA_{\pm}(\fg)$}\label{SVform}
Fix $\beta\in\Delta_{iso}$. Let  $\ft=\ft'\times\ft''$ be a decomposition with
$\dim \ft''=2$,  $\beta\in (\ft'')^*$ and
$ (\ft')^*\subset\{\mu|\ 
\langle\mu,h_{\beta}\rangle=0\}$. 
Fix  $\omega\in (\ft'')^*$ satisfying $\langle\omega,h_{\beta}\rangle\not=0$.
Since $\langle\beta,h_{\beta}\rangle=0$, the elements
 $\omega,\beta$ form a basis of  $(\ft'')^*$.

We set $x:=e^{\omega}$, $y:=e^{\beta}$ and 
write $f\in \mathbb{Z}[t^*]$ in the form
$$f=\sum_{a,b\in\mathbb{C}} f_{a,b} x^a y^b$$
where $f_{a,b}\in \mathbb{Z}[(\ft')^*]$.

\subsubsection{}
\begin{lem}{}
Assume that 
$\Delta_{iso}=W\beta\cup (-W\beta)$. Then
\begin{equation}\label{SVpartial}
\cA_{\pm}(\fg)=
\{f\in R_{\pm}(\fg)^W\ |\ 
\frac{\partial f}{\partial x}\in \mathbb{Z}[\ft^*](y\pm 1)\}.
\end{equation}
\end{lem}
\begin{proof}
Take  $f\in \mathbb{Z}[\ft^*]$ and write
$f=\sum m_{\nu} e^{\nu}=\sum_{a,b\in\mathbb{C}} f_{a,b} x^a y^b$,
 where 
$$f_{a,b}=\sum_{\nu'\in (\ft')^*} m_{\nu'+a\omega+b\beta} e^{\nu'}.$$
 Then
$\frac{\partial f}{\partial x}=\sum_{a,b\in\mathbb{C}}a f_{a,b} x^{a-1} y^b$.
The condition $\frac{\partial f}{\partial x}\in \mathbb{Z}[\ft^*](y-1)$ is equivalent to
$$\forall b\ \ \ \ \sum_i  a  f_{a,b+i} x^{a-1}=0$$
that is for each $a,b$ one has
$$0=a\sum_i  f_{a,b+i}=
a\sum_{\nu'\in (\ft')^*}  m_{\nu'+a\omega+b\beta} e^{\nu'}.$$
Since $\langle\nu'+a\omega,h_{\beta}\rangle\not=0$ if and only if $a=0$ we get
$$\{f\in\mathbb{Z}[\ft^*]|\ 
\frac{\partial f}{\partial x}\in \mathbb{Z}[\ft^*](y-1)\}=
\{\sum_{\nu} m_{\nu} e^{\nu}|\ \langle\nu,h_{\beta}\rangle\not=0
\ \Longrightarrow\  \sum_i  m_{\nu+i\beta}=0\}.$$
Similarly
$$\{f\in\mathbb{Z}[\ft^*]|\ 
\frac{\partial f}{\partial x}\in \mathbb{Z}[\ft^*](y+1)\}=
\{\sum_{\nu} m_{\nu} e^{\nu}|\ \langle\nu,h_{\beta}\rangle\not=0
\ \Longrightarrow\  \sum_i  (-1)^i m_{\nu+i\beta}=0\}.$$
\end{proof}

\subsubsection{Example}
Take $\fg=\fgl(m|n)$, $\beta:=\delta_n-\vareps_m$ and $\omega:=\vareps_m$.
Since $\fh=\ft$ we have $R(\fg)=\mathbb{Z}[P_0]$ and~(\ref{SVpartial}) gives
$$\cA_{\pm}(\fg)=
\{f\in \mathbb{Z}[P_0]^W\ |\ 
\frac{\partial f}{\partial x}\in \mathbb{Z}[\ft^*](y\pm 1)\}
$$
where $x:=e^{\vareps_m}$ and $y:=e^{\delta_n-\vareps_m}$. 
Set $x_m:=e^{\vareps_m}=x$ and
$y_n:=e^{\delta_n}=xy$. Then
$$\frac{\partial f}{\partial x}=\frac{\partial f}{\partial x_m}+\frac{y_n}{x_m}
\frac{\partial f}{\partial y_n}.$$
Since $x_m$ is invertible in $\mathbb{Z}[\ft^*]$,
 the condition $\frac{\partial f}{\partial x}\in \mathbb{Z}[\ft^*](y-1)$
can be rewritten as 
$$x_m\frac{\partial f}{\partial x_m}+y_n
\frac{\partial f}{\partial y_n}\in \mathbb{Z}[\ft^*](x_m-y_n)$$
which is the Sergeev-Veselov condition for $\cA_-(\fgl(m|n))$.

\subsubsection{Hoyt-Reif description  of $\cA_-(P)$}\label{ev-1}
Assume that $P\subset (\ft')^*+\mathbb{Z}\beta+\mathbb{Z}\omega$. 
Then 
$$\mathbb{Z}[P;\xi]\subset \mathbb{Z}[(\ft')^*;\xi]\otimes \mathbb{Z}[x^{\pm 1},y^{\pm 1}].$$
For each $c\in\mathbb{C}^*$ we denote by $\ev_z$ the evaluation map
$\ev_c:\mathbb{Z}[P;\xi]\to \mathbb{C}[(\ft')^*;\xi]$ given by
$x\mapsto c$ and $y\mapsto -1$.

\subsubsection{}
\begin{cor}{corApev}
Assume that $\Delta_{iso}=W\beta\cup (-W\beta)$.
If $P\subset (\ft')^*+\mathbb{Z}\beta+\mathbb{Z}\omega$, then
$\cA_-(P)=\{f\in R[P]^W|\ \ev_{c}(f)\ \text{ does not depend on } c\}.
$
\end{cor}
\begin{proof}
Take  $f=\sum m_{\nu} e^{\nu}\in \mathbb{Z}[P]$.
For each $\nu'\in (\ft')^*$
and $j\in\mathbb{Z}$    set
$$n_{\nu';c}:=\sum_{i}  m_{\nu'+j\omega-i\beta}.$$
The condition $\sum_{i\in\mathbb{Z}} 
 m_{\nu+i\beta}=0$ for $\langle\nu,h_{\beta}\rangle\not=0$ means that
$n_{\nu';j}=0$ if $j\not=0$.
Since $\ev_{c}(f)=\sum_{\nu',j}n_{\nu';j} c^j$,  
the above condition holds if and only if
$\ev_{c}(f)$ does not depend on $c$.
\end{proof}

\subsection{Example: $\cA(\mathfrak{p}_n)$}\label{pn}
In this section we will deduce the formula $\cA(\mathfrak{p}_n)=\Ch(\mathfrak{p}_n)$
form the results of~\cite{IRS}.
It would be interesting to obtain this formula by the method used in Section~\ref{sect1}.
\subsubsection{}
Set $\fg:=\fp_n$. One has 
$\fg_{\ol{0}}=\fgl_n$ and $\fh=\ft$, so $R(P_0)=\mathbb{Z}[P_0;\xi]$. We have
$$\Delta_{\ol{0}}^+=\{\vareps_i-\vareps_j\}_{1\leq i<j\leq n},\ \ \
\Delta_{iso}=\{\pm (\vareps_i+\vareps_j)\}_{1\leq i<j\leq n},\ \ \
\Delta_{\ol{1}}=\{2\vareps_i\}_{i=1}^n\coprod \Delta_{iso}.$$
The group $W=S_n$ acts on $\ft^*$ by permuting $\Delta^+$.
We set 
$$\str:=\sum_{i=1}^n\vareps_i,\ \  \  \beta:=\vareps_{n-1}+\vareps_n.$$
One has  $P_0=\mathbb{C}\str+ P_{int}$.
Since $\Delta_{iso}=W (\beta)\coprod W (-\beta)$
we have
$$\cA(\fg)=\{\sum_{\nu} m_{\nu} e^{\nu}\in \mathbb{Z}[P_0;\xi]^W\ |\   \langle\nu,h_{\beta}\rangle\not=0\ \Longrightarrow\ 
 \sum_{i\in\mathbb{Z}} 
(-\xi)^i m_{\nu+i\beta}=0\}.$$

\subsubsection{}
\begin{lem}{lemP}
Denote by $\mathbb{Z}[\mathbb{C}\str]$  the group ring of $\mathbb{C}\str$.
Then
$$\Ch(\fg)=\mathbb{Z}[\mathbb{C}\str]\otimes \Ch(P_{int}),\ \ \ \ \cA(\fg)=\mathbb{Z}[\mathbb{C}\str]\otimes \cA(P_{int}).$$
\end{lem}
\begin{proof}
The first formula follows from the fact that
every finite-dimensional module $M$
can be obtained from  $M\in \cF(P_{int})$ by tensoring with
one dimensional  module $L(c\str)$.
For the second formula write $a\in \mathbb{Z}[P_0;\xi]$ in the form
$$a=\sum_{\nu} m_{\nu} e^{\nu}=
\sum_{c\in\mathbb{C}/\mathbb{Z}} e^{c\str}  a_c\ \text{ with } a_c\in \mathbb{Z}[P_{int},\xi]$$
One readily sees that $a\in\cA(\fg)$ if and only if  $a_c\in\cA(\fg)$ for all $c$.
This gives the second formula.
\end{proof}

\subsubsection{}
By~\cite{IRS} we have $\Ch_-(P_{int})=\cA_-(P_{int})$. Using~\ref{A+-iso},
 (\ref{corAp+-}) and~\Lem{lemP} we get
  $\Ch_+(P_{int})=\cA_+(P_{int})$, 
  $\Ch(P_{int})=\cA(P_{int})$ and, finally,
  $\Ch(\fg)=\cA(\fg)$.

\section{Appendix}
In this section we recall a description of the set of dominants weights
for a Kac-Moody superalgebra. The results of this section can be found in~\cite{CW}, \cite{GS},  \cite{KLie}, \cite{Mus} and other sources.

\subsection{General results}
Let $\fs$ be a Lie superalgebra 
 and let $V$ be a $\fs$-module $V$. 
 
 \subsubsection{Notation}
 For $a\in\fs$ we set
 $V^a:=\{v\in V|\ av=0\}$.
We say that $a$ acts {\em locally finitely} (resp., {\em locally nilpotently})
on $V$
if any $v\in V$ lies in a finite-dimensional $a$-invariant subspace of $V$
(resp., $a^sv=0$ for $s>>0$). 
We say that $V$ is {\em $\fs$-locally finite } if  any $v\in V$ 
lies in a finite-dimensional $\fs$-submodule of $V$.
The following important proposition holds for Lie algebras, but does not hold for Lie superalgebras.

\subsubsection{Proposition (\cite{Kbook}, Prop. 3.8)}
{\em Let $\fs$ be a finite-dimensional Lie algebra and $F_V\subset\fs$
be the set of elements in $\fs$ which act locally finitely on a $\fs$-module $V$.
If $F_V$ generates $\fs$ (as a Lie algebra), then

\begin{enumerate}
\item
$F_V$ spans $\fs$;

\item
 $V$ is $\fs$-locally finite.
\end{enumerate}}

\subsubsection{Example}
The Lie superalgebra $\fosp(2n|2n)$ admit a base $\Sigma$ consisting of isotropic roots and is generated by $e_{\pm\alpha}$ with $\alpha\in\Sigma$.
Since $\alpha$ is isotropic, $e_{\pm\alpha}^2=0$, so
each $e_{\pm\alpha}$ acts nilpotently on any 
 $\fosp(2n|2n)$-module. However, if $V$ is an infinite-dimensional 
simple module, $V$ is not locally finite.

\subsubsection{}
\begin{cor}{coragi}
Let $\fs$ be a finite-dimensional Lie superalgebra and $\ft\subset\fs_{\ol{0}}$
be a commutative subalgebra. Let $V$ be a simple $\fs$-module where $\ft$ acts diagonally. Set
$$F:=\{a\in\fs_{\ol{0}}|\  a\ \text{ is $\ad$-nilpotent and }\ V^a\not=0\}.$$
If $F+\ft$ generates $\ \fs_{\ol{0}}$, then $V$ is finite-dimensional.
\end{cor}
\begin{proof}
Take  $a\in F$ and a non-zero $v\in V^a$. Since $a$
 acts locally nilpotently on $\cU(\fs)$, $a$ acts locally nilpotently on
$\cU(\fs)v$. Since $V$ is simple, $a$ acts locally nilpotently on $V$.
By (ii),  any cyclic  $\fs_{\ol{0}}$-submodule of $V$ is finite-dimensional.
By PBW Theorem, $V$ is finitely generated over $\fs_{\ol{0}}$;
hence $V$ is finite-dimensional.
\end{proof}

\subsection{Applications to Kac-Moody case}\label{KM}
From now on $\fg$ is a 
 finite-dimensional  Kac-Moody superalgebra or $\fgl(m|n)=\fsl(m|n)\times\mathbb{C}$. The algebra $\fg_{\ol{0}}$ is reductive.

 \subsubsection{Notation}
We retain notation of~\ref{tri}. In our case $\fh=\ft$
and $\fg$ admits a non-degenerate invariant bilinear form;
we denote by $(-|-)$ the corresponding form on $\ft^*$.  

Recall that 
$\Delta_{iso}$ stands for
the set of roots $\beta$ such that for some $e_{\pm}\in \fg_{\beta}$
the subalgebra spanned by $e_{\pm}, h_{\beta}:=[e_+,e_-]$ is isomorphic to
 $\fsl(1|1)$. One has 
$\Delta_{iso}=\{\beta\in\Delta|\ (\beta|\beta)=0\}$.

 In our  case 
 $\Delta_{\ol{0}}\cap\Delta_{iso}=\emptyset$.
Moreover, for each $\alpha\in\Delta_{\ol{0}}$
the algebra generated by $\fg_{\alpha},\fg_{-\alpha}$
is isomorphic to $\fsl_2$; we denote 
 by $\alpha^{\vee}$ the coroot
corresponding to $\alpha$.

 A set $\Sigma\subset\Delta^+$ is called a {\em base} if the elements
of  $\Sigma$ are linearly independent and $\Delta^{\pm}$ lies
in $\pm\mathbb{N}\Sigma$ (where $\mathbb{N}\Sigma$
is the set of non-negative integral linear combinations of $\Sigma$);
in this case we write $\Delta^+=\Delta^+(\Sigma)$.
For the superalgebras which we consider each set of positive roots admits
a base. We denote by $\cB$ the set of bases compatible with the fixed triangular
decomposition of $\fg_{\ol{0}}$, i.e., $\Sigma\in\cB$ if
$\Delta_{\ol{0}}^+\subset\Delta^+(\Sigma)$. 

For each $\Sigma\in\cB$ we set
$\Sigma_{iso}=\Sigma\cap \Delta_{iso}$.
 For each $\alpha\in\pi$  we  set
$$\cB_{\alpha}:=\{\Sigma\in \cB| \ 
\alpha\in\Sigma \text{ or }\frac{\alpha}{2}\in\Sigma\}.$$

\subsubsection{Odd reflections}\label{odd}
Given a subset of positive roots $\Delta^+$
(containing $\Delta_{\ol{0}}^+$)
and  $\beta\in\Sigma_{iso}$, we can construct a new subset of positive roots
(containing $\Delta_{\ol{0}}^+$) by an {\em odd reflection} $r_{\beta}$:
$$
r_{\beta}(\Delta^+)=(\Delta^+\setminus\{\beta\})\cup\{-\beta\}.$$

We denote by $r_{\beta}\Sigma$ the base of $r_{\beta}\Delta^+(\Sigma)$.

\subsubsection{Proposition~\cite{Sint}}\label{propS}
\begin{enumerate}
\item
Any two subsets of positive roots in $\Delta$ (containing $\Delta_{\ol{0}}^+$)
can be obtained from each other by a finite sequence of odd reflections.

\item The set $\cB_{\alpha}$ is non-empty for  each $\alpha\in\pi$.

\item $r_{\beta}\Sigma$
 contains $-\beta$, the roots $\alpha\in\Sigma\setminus\{\beta\}$,
which are orthogonal to $\beta$, and the roots $\alpha+\beta$ for $\alpha\in\Sigma$ which are not orthogonal to $\beta$. 
\end{enumerate}

\subsubsection{Highest weight modules}
 For any $\Sigma\in\mathcal B$ we denote by $L(\lambda;\Sigma)$ the irreducible module of highest weight $\lambda$ with respect to the Borel
subalgebra corresponding to $\Sigma$.
For an odd root
$\beta\in\Sigma$  one has
\begin{equation}\label{lambdalambda'}
 L(\lambda;\Sigma)=\left\{\begin{array}{ll} L(\lambda;r_{\beta}\Sigma) & \text{ if } \langle\lambda,h_{\beta}\rangle=0\\
L(\lambda-\beta;r_{\beta}\Sigma) &  \text{ if } \langle\lambda,h_{\beta}\rangle\not=0.
\end{array}\right.\end{equation}
In particular, if $L$ is a simple highest weight module with respect to
some $\Sigma\in\cB$, then $L$ is a simple highest weight module with respect to
any $\Sigma'\in\cB$. We denote by $\hwt_{\Sigma} (L)$ the highest weight of $L$
with respect to $\Sigma$.

\subsection{}
\begin{prop}{propi}
For a simple highest weight $\fg$-module $L$
the following conditions are equivalent
\begin{itemize}
\item[(a)]
$L$ is finite-dimensional;
\item[(b)]
for each $\alpha\in \pi$ there exists $\Sigma\in \cB_{\alpha}$ such that
$\hwt_{\Sigma} (L)\in P^+(\pi)$;
\item[(c)] for each $\alpha\in \pi$ there exists $\Sigma\in \cB_{\alpha}$ such that
$\langle \hwt_{\Sigma} (L),\alpha^{\vee}\rangle \in\mathbb{N}$.
\end{itemize}
\end{prop}
\begin{proof}If $L$ is finite-dimensional, then for each base $\Sigma\in\cB$ one has
 $\hwt_{\Sigma} (L)\in P^+(\pi)$ 
and so
$\langle \hwt_{\Sigma} (L),\alpha^{\vee}\rangle \in\mathbb{N}$
for each $\alpha\in\pi$. This gives the implications
(a) $\Longrightarrow $ (b)  $\Longrightarrow $ (c).

For the implication (c) $\Longrightarrow $ (a) we
assume that for each $\alpha\in \pi$ there exists $\Sigma\in \cB_{\alpha}$ such that
$\langle \hwt_{\Sigma} (L),\alpha^{\vee}\rangle \in\mathbb{N}$. By~\Cor{coragi}, it is enough to show that the set
$$F:=\{a\in\fg_{\ol{0}}\ |\  a\ \text{ is $\ad$-nilpotent and }\ L^a\not=0\}.$$
contains
$\fg_{\pm\alpha}$ for each $\alpha\in\pi$. 
Fix $\alpha\in\pi$ and $\Sigma\in \cB_{\alpha}$ such that 
for $\nu:=\hwt_{\Sigma}(L)$ one has
$\langle \nu,\alpha^{\vee}\rangle \in\mathbb{N}$.  

Since $L$ is a highest weight module, $F$ contains $\fg_{\alpha}$.
Fix a highest weight  vector $v\in L_{\nu}$. 
Set $\alpha':=\alpha$ if $\alpha\in\Sigma$ and $\alpha':=\alpha/2$ if $\alpha/2\in\Sigma$.
Denote by $\fl$ the subalgebra generated by $\fg_{\pm\alpha'}$.
One has $\fl\cong \fsl_2$ if $\alpha'=\alpha$ and $\fl\cong \fosp(1|2)$ if $\alpha'=\alpha/2$; it is well-known that 
a simple highest weight $\fl$-module $L_{\fl}(\mu)$ is finite-dimensional
if and only if $\langle \mu,\alpha^{\vee}\rangle \in\mathbb{N}$. 
Fix non-zero $e_{\pm}\in\fg_{\alpha'}$.
Since  $\langle\nu, \alpha^{\vee}\rangle \in\mathbb{N}$ one has
$\dim L_{\fl}(\nu)<\infty$, so
$e_+ e_-^kv=0$ for some $k>0$.  For each $\beta\in\Sigma\setminus\{\alpha'\}$
one has $[\fg_{\beta},e_+]=0$, so  $\fg_{\beta}(e_-^kv)=0$.
Hence  $e_-^kv$ is a primitive vector. Since $L$ is simple, $e_-^kv=0$.
Since $\fg_{-\alpha}$ is spanned either by $e_-$  or by $e_-^2$, 
$\fg_{-\alpha}$ acts nilpotently on $v$, so  
 $\fg_{-\alpha}\subset F$ as required.
\end{proof}

\subsubsection{}
\begin{cor}{coro}
Let $\Sigma\in\cB$ be such that for each $\alpha\in\pi\setminus\Sigma$ one has  $r_{\beta}\Sigma\in \cB_{\alpha}$ for some $\beta\in\Sigma_{iso}$. Then
\begin{equation}\label{P+nicy}
P^+(\Delta^+)= \{\lambda\in P^+(\pi)|\ \forall \beta\in\Sigma\cap\Delta_{iso}\ \
\langle\lambda,h_{\beta}\rangle\not=0\ \Longrightarrow \lambda-\beta\in P^+(\pi)\}.\end{equation}
\end{cor}
\begin{proof}
The assertion follows from~\Prop{propi} and~(\ref{lambdalambda'}).
\end{proof}

\subsubsection{}
\begin{cor}{cora}
If $\Sigma$ satisfies the property 
\begin{equation}\label{Pr2}
\text{ each $\alpha\in\pi$ is the sum of at most two elements in $\Sigma$}
\end{equation}
then  the assumption of~\Cor{coro} holds for $\Sigma$.\end{cor}
\begin{proof}
It is enough to check that
if $\alpha\in\pi$ is the sum of two different roots from $\Sigma$, then 
$\alpha\in r_{\beta}\Sigma$ for some $\beta\in\Sigma_{iso}$.
Write $\alpha=\beta+\beta'$ for $\beta\not=\beta'\in\Sigma$.
Since $\alpha\in\pi$, the roots $\beta$, $\beta'$ are odd.

Assume that $\beta\in\Delta_{iso}$. If $(\beta|\beta')\not=0$, then
$\alpha\in r_{\beta}\Sigma$ as required.
If $(\beta|\beta')=0$, then
$-\beta,\beta'\in r_{\beta}\Sigma$, so $\alpha\not\in\Delta^+(r_{\beta}\Sigma)$
which contradicts to $\Delta^+(r_{\beta}\Sigma)=\Delta^+(\Sigma)\setminus\{\beta\}\cup\{-\beta\}$.

If $\beta,\beta'\not\in\Delta_{iso}$ are  non-isotropic, then $2\beta,2\beta'\in\Delta_0^+$
(since $\beta,\beta'$ are odd)
and $2\beta-2\beta'=2\alpha$ which contradicts to $\alpha\in\pi$.
\end{proof}

\subsection{Dynkin diagrams}
The root systems corresponding to $\fg$ 
are called:
$A(m|n)$ (for $\fg=\fgl(m|n)$ or $\fg=\fsl(m|n)$ with $m\not=n$),
$B(m|n)$ (for $\mathfrak{osp}(2m+1|2n)$), $D(m|n)$ (for $\mathfrak{osp}(2m|2n)$),
$G(3)$ and $F(4)$. The root system of $D(2|1;a)$ coincides with the root system
$D(2|1)$. The root systems $D(1|n)$ is also called $C(n+1)$.
The root systems $A(m|n), D(1|n)$ are called {\em type I}
systems; the rest of roots systems are called {\em type II} systems.
The root system $\Delta$ is of type II if and only if $\fg_{\ol{0}}$ is semisimple.

The Dynkin diagram of $\Sigma\in\cB$ was introduced in~\cite{KLie}.
Except for $A(m|n)$ case there is a one-to-one correspondence between
$\cB$ and the sets of Dynkin diagram. For $A(m|n)$ case
the set $\cB$ admits an involution $-w_0$ and  $\Sigma$, $-w_0\Sigma$ have the same Dynkin diagram ($w_0$ is the longest element in the Weyl group).

\subsubsection{Mixed bases}
We say that a Dynkin diagram is {\em mixed} if 
this diagram contains 
\begin{itemize}
\item
as many ``$\bullet$''s as possible, 
\item as less ``$\circ$''s as possible,
\item
 as many cycles containing the edge $\otimes--\otimes$ as possible.
\end{itemize}
We say that a base $\Sigma\in\cB$ is  {\em mixed} if its Dynkin diagram is mixed. The mixed Dynkin  diagram is
unique; a mixed base is unique 
if $\Delta$ is not of type $A(m|n)$ or has type $A(n+1|n)$. For $A(m|n)$ with $m\not=n\pm 1$ there are
two mixed bases: $\Sigma$ and $-w_0\Sigma$.

For $A(n|n), A(n+1|n), B(n|n), D(n+1|n), D(n|n)$ the mixed bases consists
of odd roots.

The mixed bases satisfy  the property~(\ref{Pr2}). Moreover
for type II systems  the mixed bases are the only bases
satisfying~(\ref{Pr2}), see~\ref{pr23} below.

\subsection{Property~(\ref{Pr2})}
Consider a simplified version of the\label{pr23} Dynkin diagram: 
we assign to each simple root $\alpha$ a node
 $\circ$ if
$\alpha$ is even,  $\scriptstyle{\otimes}$ if $\alpha$ is isotropic
and  $\bullet$ if $\alpha$ is odd and non-isotropic; we connect 
two nodes $\alpha,\beta$ if $\alpha+\beta\in\Delta$ (it is easy to see 
that $\alpha+\beta\in\Delta$ if and only if
$(\alpha|\beta)\not=0$). Denote the diagram
by $\Dyn$ and denote by $\Dyn_{odd}$ the subdiagram which contains only
$\scriptstyle{\otimes}$, $\bullet$ and the edges between them.

\subsubsection{Remark}
For $F(4)$ there are two bases in $\cB$ which correspond to
the same simplified Dynkin diagram $\circ-\scriptstyle{\otimes}-\circ-\circ$.

\subsubsection{}
Below we list the simplified Dynkin diagrams for mixed bases.

For $A(m|n)$  the diagram is  $\scriptstyle{\otimes}-\otimes-\ldots\otimes-\circ-\dots-\circ$
with $m+n-1$ nodes where the number $\#(\circ)$ 
of $\circ$ nodes  is $0$ for  $m=n,n\pm 1$  and is $m-n-1$  if $m>n$.

For $B(m|n)$ the   diagram is
 $\circ-\dots-\circ-\scriptstyle{\otimes}-\otimes-\dots-\bullet$, 
where $\#(\circ)=|m-n|$ and the total number of nodes is $m+n$.
For $G(3)$ we have a similar simplified Dynkin diagram 
with three nodes and $\#(\circ)=1$.

For $D(1|n)=C(n+1)$ the   diagram contains $1+n$ nodes and takes the form
 $$
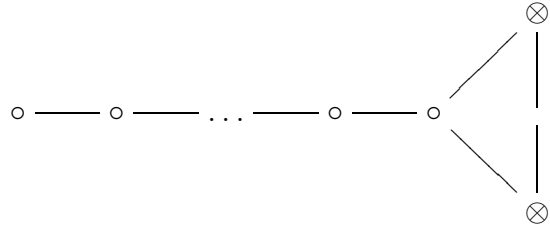
\xymatrix{ &  & & & & \otimes\ar@{-}[d]\\
\circ\ar@{-}[r]&\circ\ar@{-}[r]&\dots\ar@{-}[r]&\circ\ar@{-}[r]&
\circ\ar@{-}[ur]\ar@{-}[dr] & \\
& & & & & \otimes\ar@{-}[u].}$$

For $D(m|n)$ with $m>1$ the   diagram contains $m+n$ nodes and takes the form
 $$
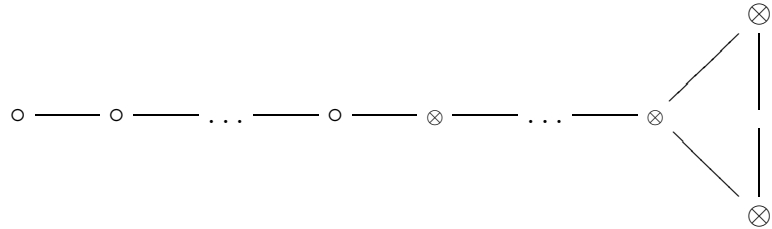
\xymatrix{ &  & & & & & & \otimes\ar@{-}[d]\\
\circ\ar@{-}[r]&\circ\ar@{-}[r]&\dots\ar@{-}[r]&
\circ\ar@{-}[r]&\scriptstyle{\otimes}\ar@{-}[r]&\dots\ar@{-}[r]&\scriptstyle{\otimes}\ar@{-}[ur]\ar@{-}[dr] & \\
& & & & & & & \otimes\ar@{-}[u]}$$
where $\#(\circ)=0$  for  $m=n,n+1$, $\#(\circ)=m-n-1$ for $m>n$
and $\#(\circ)=n-m$ for $n>m$. For $F(4)$ we have a similar  diagram 
with four nodes and $\#(\circ)=1$.

\subsubsection{}
Let $\ell$ (resp., $\ell'$) be the cardinality of $\pi$ (resp., of $\Sigma$).

It is easy to see that for two odd simple roots $\beta,\beta'$
with $\beta+\beta'\in\Delta$ one has $\beta+\beta'\in\pi$.
As a result  the number of roots in $\pi$ which can be written as
$\beta+\beta'$ for 
$\beta,\beta'\in\Sigma$ is equal to the number of edges in $\Dyn_{odd}$
which we denote by $\#(odd\ edge)$
plus the number of $\bullet$s which we denote by $\#(\bullet)$.
Since $\ell'=\#(\circ)+\#(\bullet)+\#(\scriptstyle{\otimes})$ 
 the property~(\ref{Pr2}) can be rewritten as 
\begin{equation}\label{Pr3}
\ell'-\ell=\#({\scriptstyle{\otimes}})-\#(odd\ edge).\end{equation}

For type I we have $\ell=\ell'-1$ and $\#(\bullet)=0$. 
The formula~(\ref{Pr3}) means that 
the number of edges in $\Dyn_{odd}$ is less by one than the
number of vertices. The diagrams $\Dyn_{odd}$ do not have cycles,
so this condition holds if and only if  $\Dyn_{odd}$ is connected.
For $C(n+1)$ the diagram $\Dyn_{odd}$  is always connected and~(\ref{Pr3}) holds. 

For type II we have $\ell=\ell'$.

For $B(m|n)$ and $G(3)$ one has $\#(\bullet)\leq 1$ and
the diagrams $\Dyn_{odd}$ do not have cycles. 
Therefore~(\ref{Pr3})
holds if and only if   $\Dyn_{odd}$ is connected and $\#(\bullet)=1$.
The mixed Dynkin diagram is the only diagram satisfing these conditions.

For the remaining cases $D(m|n)$ and $F(4)$ one has
 $\Delta_1=\Delta_{iso}$, so $\#(\bullet)=0$. 
From the classification
it follows that $\Dyn$ may contain at most one cycle.
Therefore~(\ref{Pr3}) 
holds if and only if  $\Dyn_{odd}$ is connected and contains a cycle.
The mixed Dynkin diagram is the only diagram satisfing these conditions.

\subsubsection{Remark}\label{netolko}
If all odd roots are isotropic, then the assumption of~\Cor{coro}
is equivalent to~(\ref{Pr2}). 
For $B(m|n)$ the assumption of~\Cor{coro}
holds for the mixed diagram and for $\Sigma'$ corresponding to  $\circ-\dots-\circ-\scriptstyle{\otimes}-\otimes-\dots-\circ$, 
where $\#(\circ)=|m-n+1|$; for instance,  $B(1|1)$
admits two diagrams $\scriptstyle{\otimes}-\bullet$ and $\scriptstyle{\otimes}-\circ$; the formula~(\ref{P+nicy})
hold in both cases.
For $G(3)$ the assumption of~\Cor{coro}
holds for the mixed diagram and for 
  the cyclic diagram.

\end{document}